\documentclass[12pt,a4paper]{amsart}

\usepackage{mathtools}
\mathtoolsset{showonlyrefs=true}
\usepackage{xcolor} % kolory
\usepackage[hypertexnames=false, colorlinks=true, linkcolor = teal, citecolor = blue]{hyperref}

\usepackage[textwidth=1.8cm]{todonotes}
\usepackage{marginnote}

\usepackage[T1]{fontenc}
\usepackage[utf8x]{inputenc}
\usepackage{microtype}

\usepackage{comment}
\newtheorem{thm}{Theorem}
\newtheorem{prop}[thm]{Proposition}
\newtheorem{lem}[thm]{Lemma}
\newtheorem{cor}[thm]{Corollary}
\newtheorem{assu}[thm]{Assumption}

\theoremstyle{definition}

\newtheorem{ex}[thm]{Example}
\newtheorem{rem}[thm]{Remark}

\numberwithin{equation}{section}
\numberwithin{thm}{section}

\newcommand{\R}{\mathbb{R}}
\newcommand{\D}{\mathcal{D}}

\renewcommand{\d}{\mathrm{d}}

\title[Concentration of mass]
{Concentration of mass of solutions to aggregation--diffusion equations}

\author[G. Karch]{Grzegorz Karch}
\address[G. Karch]{Instytut Matematyczny, Uniwersytet Wroc{\l}awski, pl. Grunwaldzki 2/4, 50-384 Wroc{\l}aw, Poland}
\email{grzegorz.karch@math.uni.wroc.pl}

\author[K. Krawczyk]{Krzysztof Krawczyk}
\address[K. Krawczyk]{Instytut Matematyczny, Uniwersytet Wroc{\l}awski, pl. Grunwaldzki 2/4, 50-384  Wroc{\l}aw, Poland}
\email{krzysztof.krawczyk@math.uni.wroc.pl}

\author[A. Lanar]{Alexandre Lanar}
\address[A. Lanar]{Universit\'e Lyon 1, Centrale Lyon, INSA Lyon, Universit{\'e} Jean Monnet, CNRS, ICJ UMR5208, 69622 Villeurbanne, France}
\email{lanar@math.univ-lyon1.fr}

\date{\today}

\begin{document}
	
\keywords{nonlocal drift-diffusion equation; small diffusivity; concentration of solutions}
	
\subjclass[2010]{35Q92; 35K55; 35B36; 35B45}

\begin{abstract}

\end{abstract}

\begin{abstract}
		We consider the aggregation-diffusion equation
\begin{equation}
    u_t - \varepsilon \Delta u = \nabla \cdot (u \nabla K \ast u)
\end{equation}
in the whole space with a \textit{mildly singular} interaction kernel $K=K(x)$ which behaves like $|x|^k$ near the origin for some $k \in (0,2)$. This equation, supplemented with nonnegative, bounded, and integrable initial data, possesses a global-in-time solution. We prove that the family of nonnegative, radially symmetric solutions of this equation, all sharing the same initial datum, focuses around the origin over a~common finite time interval as $\varepsilon \searrow 0$.
\end{abstract}

%%%%%%%%%%%%%%%%%%%%%%%%%%%%%%%%%%%%%%%%%%%%%%
\maketitle

\section{Introduction}\label{sec:int}

\subsection{Aggregation-diffusion models.}
The non-local nonlinear evolution model
\begin{equation}\label{eq:1}
u_t - \varepsilon \Delta u = \nabla \cdot (u \nabla K \ast u), \quad x \in \mathbb{R}^d, \ t > 0,
\end{equation}
describes the evolution of a particle density $u = u(x, t)$ interacting through pairwise potentials, characterized by convolution with a kernel $K : \mathbb{R}^d \to \mathbb{R}$. Such models are prominent in various disciplines. 
In astrophysics, mean-field models for gravitationally attracting particles are often rooted in the Chandrasekhar equation for stellar equilibrium \cite{Chandrasekhar1942,Chavanis1996}. 
Similarly, in mathematical biology, variations of the Keller-Segel system are utilized to describe population dynamics in phenomena such as chemotaxis, haptotaxis, and angiogenesis \cite{HillenPainter2009}:
\begin{equation}\label{KS}u_t - \Delta u = -\nabla \cdot (u \nabla v), \quad \Delta v + u = 0.\end{equation}
By selecting the fundamental solution of $-\Delta$ as the aggregation kernel $K(x)$, system~\eqref{KS} in $\mathbb{R}^d$ reduces to equation \eqref{eq:1}.

For a comprehensive survey of recent results on equation \eqref{eq:1} (involving linear or nonlinear diffusion),
we refer to \cite{Carrillo2019,Gomez-Castro24} and the references therein. 
In the specific case of linear diffusion, equation \eqref{eq:1} is often referred to as the McKean–Vlasov equation, which is associated with a stochastic differential equation via a mean-field limit procedure \cite{Jabin2017,Bresch2019}.

% In prototypical processes such as the collective behavior of the slime mold {\it Dictyostelium Discoideum}, the signal
% is produced by the cells themselves, and the possibly most striking consequence thereof appears to be the ability
% of cell populations to spontaneously form aggregates in small spatial regions after a finite time.

\subsection{Blowup of solutions.}\label{subsection:blowup}
Model \eqref{eq:1} (and, in particular the Keller-Segel system \eqref{KS}) is mathematically significant primarily due to the potential for its solutions to develop singularities in either finite or infinite time. The mathematical research here focuses on identifying the specific constraints on initial data and system parameters that either guarantee or preclude the occurrence of blow-up phenomena.
Given the vast body of literature concerning blow-up in model \eqref{eq:1}, we recall below only some foundational results.

In the Keller-Segel model \eqref{KS} on the plane, the mass $M = 8\pi$ serves as a critical threshold for the solution's behavior. If the initial condition is a non-negative measure with a finite second moment and total mass $M < 8\pi$, the system admits a unique global-in-time solution. However, for $M > 8\pi$, the solution cannot be extended globally as a regular solution, and a finite-time blow-up occurs \cite{BedrossianMasmoudi2014, BilerEtAl2006, BilerZienkiewicz2015, BlanchetDolbeaultPerthame2006}.
In the critical case $M = 8\pi$, and under specific technical assumptions (such as the finiteness of the second moment), the global solution $u(x, t)$ undergoes infinite-time concentration, \cite{blanchet2008infinite}.

For dimensions $d \geq 3$, the initial value problems for \eqref{eq:1} and \eqref{KS} typically possess local-in-time solutions in various functional spaces. Furthermore, there are global-in-time solutions for small initial data, as well as blow-up solutions for large initial data, for both the Keller-Segel model \eqref{KS} and the aggregation-diffusion equation \eqref{eq:1} with a kernel $K(x)$ that is unbounded at the origin,
see {\it e.g.}~\cite{CPZ04,KarchSuzuki2011, CCE12,LemarieRieusset2013}.

%%%%%%%%%%%%%%%%%%%%%%%%%%%%%

\subsection{Concentration of mass in this work} 
 The goal of this work is to show that concentration phenomena within the general model \eqref{eq:1} are not limited to the blow-up of solutions; they can also be observed in the context of global-in-time regular solutions.
More precisely, we consider the kernel $K=K(x)$ in problem~\eqref{eq:1} that behaves like $|x|^k$  in a neighborhood of zero
for some $k\in (0,2)$ (see Assumptions \ref{ass:gen}-\ref{ass:h} and Example \ref{example}, below). 
Under our assumptions imposed on the kernel $K=K(x)$, the Cauchy problem for equation \eqref{eq:1}, supplemented with nonnegative, bounded, and integrable initial conditions, has a global-in-time solution whose $L^p$-norms are bounded uniformly in time $t>0$
(see Section \ref{sec:git}).
However, 
we prove in Theorem~\ref{thm:main} below that the solutions $u_\varepsilon=u_\varepsilon(t,x)$ to equation \eqref{eq:1} with the same nonnegative, radial, integrable and bounded initial datum satisfy
\begin{equation}\label{main_estimate:0}
\int_0^T \int_{B_{(\nu \varepsilon)^{1/k}}}
u_\varepsilon(t,x)
\, \d x \, \d t\ge C \qquad\text{for all} \quad \varepsilon\in (0, \varepsilon^*).
\end{equation}
for some constants $T>0$, $\nu>0$, $C>0$, $\varepsilon^*>0$, independent of $\varepsilon$.
This result indicates that even if the considered kernel $K$ does not lead to the formation of singularities for the solution $u_\varepsilon$ for any fixed $\varepsilon > 0$, 
a concentration phenomenon can still occur in  the  balls $B_{(\nu \varepsilon)^{1/k}}$ for sufficiently small $\varepsilon > 0$. Moreover, one can immediately deduce from inequality~\eqref{main_estimate:0} (see also Remark \ref{rem:Linfty} below) that for a certain $T>0$
\[
    \sup_{t \in [0,T],\; |x|\le (\nu \varepsilon)^{1/k}} |u_\varepsilon(t,x)| \to +\infty \quad \text{as} \quad \varepsilon \searrow 0.
\]
Our results extend recent studies \cite{BBKL21, BBKL22}, which were restricted to kernels exhibiting $|x|$-like behavior at the origin.

We emphasize that the concentration of solutions described above holds on a common finite interval $[0, T]$ 
and can be interpreted as the {\it intermediate asymptotics} of the solutions as $\varepsilon \searrow 0$. We expect that, under the assumptions of this work, the $L^p$-norms of the solutions decay as $t \to \infty$. Results regarding such decay (specifically, the self-similar large-time behavior) for solutions to the Cauchy problem for equation \eqref{eq:1} with regular kernels have been published in \cite{Carrillo-similification}. The temporal decay of solutions to a one-dimensional version of model \eqref{eq:1} with less regular kernels, under a smallness assumption, was obtained in \cite{KS10}.

%%%%%%%%%%%%%%%%%%
\subsection{Related work on the aggregation equation}
The aggregation equation (the non-local equation without dissipation) given by
\begin{equation}\label{eq:2}
u_t = \nabla \cdot (u \nabla K \ast u), \quad x \in \mathbb{R}^d, \ t > 0,
\end{equation}
has been extensively studied in the fundamental papers by Bertozzi, Laurent  {\it et al.} \cite{BCL09, BGL12, BL07, BLR11}. These works establish that the regularity of the kernel $K$ is essential for the global-in-time existence of solutions. If $K$ is at least a $C^2$ function, then no finite-time blow-up occurs and solutions are global in time. On the other hand, for radial kernels satisfying the Osgood condition:
\begin{equation}\label{Osgood}
K(x) = g(|x|) \quad \text{and} \quad \int_0^1 \frac{1}{g'(r)}\,dr < \infty,
\end{equation}
solutions to the Cauchy problem for equation \eqref{eq:2} may cease to exist in finite time, at least for bounded and compactly supported initial data. In this regime, interactions are sufficiently strong to trigger finite-time blow-up, as discussed in \cite{BCL09, BGL12, BL07, BLR11} and the references therein. Note that for 
$K(x) \sim |x|^k$ in the neighborhood of the origin, assumption \eqref{Osgood} holds true for $k < 2$.

Moreover, such a solution to equation \eqref{eq:2} concentrates into a Dirac delta function at the moment of blow-up (see, e.g., \cite{BGL12, CDFLS11}). This behavior is consistent with our main result expressed by inequality \eqref{main_estimate:0}, stating that for a small diffusion coefficient $\varepsilon > 0$, an $\varepsilon$-uniform portion of the total mass of the solutions to equation \eqref{eq:1} is concentrated within an $\varepsilon$-neighborhood of the origin.
 
 % Regarding the positivity of solutions and the smoothing effect of diffusion in the more regular case ($k=1$), we refer the reader to \cite{LR09, LR10}. In our current setting, the assumptions on the initial data are slightly weaker than those in the work of Karch and Suzuki; specifically, we consider data in weak Lebesgue spaces rather than standard Lebesgue spaces. While some related results appear in the work of Laflèche and Salem \cite{LS21}, their analysis does not cover the full range of parameters addressed here. Furthermore, our approach provides a more direct treatment of moment estimates, in contrast to the $|x|$-weighted moments considered in works such as \cite{BKL09}.

% %%%%%%%%%%%%%%%%%%%%%%%%%%%%

%%%%%%%%%%%%%%%%%%%%%%%%%%%%%
\subsection{Classification of concentration phenomena}
The results of this work allow us to complete the classification of concentration phenomena for solutions to equation \eqref{eq:1}, depending on the behavior of the kernel $K(x)$ at the origin. 
Specifically, consider kernels of the form
$$ K(x) = \frac{|x|^k}{k} \quad \text{in a neighborhood of} \quad x=0$$
where, for $k=0$, we formally set $K(x) = \log|x|$. 
We distinguish the following three cases:

\begin{enumerate}
    \item If $k \in (-d, 0]$, there exist integrable, bounded, and non-negative initial conditions such that the corresponding solutions to the Cauchy problem for equation \eqref{eq:1} blow up in finite time. This blow-up phenomenon was briefly recalled 
    in Subsection~\ref{subsection:blowup}.
    %\retka{Do you mean in the end of Section 1.2.?}

    \item For $k \in (0, 2)$, we observe the concentration of solutions to the Cauchy problem as $\varepsilon \searrow 0$. This phenomenon, which is established and discussed in this paper, is closely related to blow-up results for the aggregation equation  \eqref{eq:2}.

    \item If $k \in [2, \infty)$, or more generally if $K(x)$ is $C^2$ in a neighborhood of the origin, no concentration of solutions is observed. In this case, the $L^p$-norms of solutions to equation \eqref{eq:1}, supplemented with the same initial datum, remain uniformly bounded as $\varepsilon \searrow 0$; see, {\it e.g.}, the estimates in \cite[Sec.~2.1]{BLR11}.
\end{enumerate}
%%%%%%%%%%%%%%%%%%%%%%%%%%%%%

%%%%%%%%%%%%%%%%%%%%%%%%%%%%%%%%%%%
\subsection{Analogous phenomena in other models}
In the study of partial differential equations, the singular perturbation limit occurs as the diffusion coefficient $\varepsilon \searrow 0$. In this regime, solutions often lose global smoothness and concentrate on lower-dimensional sets, such as points (spikes), curves (interfaces), or surfaces (boundary layers).

In this context, the proofs presented in this work rely on a novel methodology developed in the works \cite{Biryuk2001, Boritchev2014, Boritchev2016, Boritchev2018}, devoted to the analysis of equations arising in fluid mechanics, with particular emphasis on the estimation of small-scale quantities—such as structure functions—commonly used in the study of hydrodynamic turbulence.
For example, for the viscous Burgers equation
$$u_t + uu_x = \varepsilon u_{xx},$$
smallness of the diffusion coefficient allows the convective term to steepen the solution profile. In the vanishing-viscosity limit, energy dissipation becomes entirely concentrated on a shock front, leading to a unique entropy solution.
A similar analytical framework has also been applied to the study of the large-scale evolution of the universe, modeled by a multidimensional analogue of the classical Burgers equation \cite{Boritchev2016}.

Another concentration phenomenon observed in a prototypical model for phase separation is given by the Allen--Cahn equation:
\begin{equation}
\partial_t u_\varepsilon = \Delta u_\varepsilon - \frac{1}{\varepsilon^2} W'(u_\varepsilon),
\end{equation}
where 
$W(u) = {4}^{-1}(1-u^2)^2$ is a double-well potential with minima at $\pm 1$. 
The parameter $\varepsilon > 0$ represents the width of the diffuse interface 
between the two phases.
As $\varepsilon \searrow 0$, the solution $u_\varepsilon=u_\varepsilon(t,x)$ converges almost everywhere 
to the pure states $\pm 1$, and the energy concentrates in a thin transition 
layer of thickness $O(\varepsilon)$ separating the phases (see {\it e.g,} \cite{Ilmanen1993, EvansSonerSouganidis1992,chen1992generation}).

The Gierer--Meinhardt reaction--diffusion system
\begin{align}
u_t &= \varepsilon^2 \Delta u - u + \frac{u^p}{v^q}, \\
\tau v_t &= D \Delta v - v + u^r,
\end{align}
posed on a bounded domain with the Neumann boundary conditions, describes the evolution of 
the activator  $u=u(x,t)$ and  the inhibitor $v=v(x,t)$. Here, the parameter $\varepsilon^2 \ll 1$ denotes small activator diffusion.
For $\varepsilon\searrow 0$, stationary
solutions of this system can develop sharply localized, high-amplitude spikes, more precisely, the activator
component concentrates and converges weakly to a sum of Dirac measures,
\[
u_\varepsilon \rightharpoonup \sum_i c_i \,\delta_{x_i}.
\]
The spike positions are typically selected by the
geometry of the domain or by extrema of an associated potential (see \cite{NiTakagi1993GM,Wei1999GM,WardWei1998GM,DelPinoKowalczykWei2005} and the references therein).

The  Schr\"odinger equation in semiclassical scaling has the form
\begin{equation}
i\hbar \,\partial_t \psi^\hbar
= -\frac{\hbar^2}{2m}\,\Delta \psi^\hbar + V\,\psi^\hbar,
\end{equation}
where the solution $\psi^\hbar=\psi^\hbar(t,x)$ is called as the wave function, $V=V(x)$ is a given potential, and
$\hbar$ is the (scaled) Planck constant.
As $\hbar \to 0$, quantum dynamics approach classical mechanics. In this
regime, the probability density $|\psi^\hbar(t,x)|^2$ tends to concentrate
along trajectories determined by the classical Hamiltonian flow associated
with the symbol
\[
H(x,\xi) = \frac{|\xi|^2}{2m} + V(x).
\]
A rigorous framework for describing this concentration is given by
\emph{Wigner measures} (or semiclassical measures), which capture the weak
limits of phase-space energy distributions of $\psi^\hbar$. These measures
propagate along classical trajectories and provide a bridge between the
quantum Schr\"odinger equation and classical transport equations in phase
space (see {\it e.g.} \cite{Gerard1991,LionsPaul1993,MarkowichMauserPoupaud1994,DimassiSjostrand1999} for more details and additional references.

%%%%%%%%%%%%%%%%%%%%%%%%%%%%%%%%%%%%%%%%%%%%%%%%%%%%%%%%%%%%%

\section{Main results} \label{sec:mr}

Our goal is to show the concentration (as $\varepsilon \searrow 0$) of global-in-time solutions to the Cauchy problem
\begin{equation}\label{eq:main}
	\begin{cases}
		u_t - \varepsilon \Delta u = \nabla \cdot (u \nabla K \ast u),
		& \quad x \in \R^d, \ t > 0,\\
		u(0,x) = u_0(x), & \quad x \in \R^d,
	\end{cases}
\end{equation}
in an arbitrary spatial dimension $d\ge 1$ and with a non-negative initial condition
\begin{equation}\label{ini:L1Linf}
u_0 \in L^1(\R^d) \cap L^\infty(\R^d), \quad u_0 \ge 0.
\end{equation}

The kernel $K=K(x)$ in problem~\eqref{eq:main} behaves like $|x|^k$  in a neighborhood of zero
for some $k\in (0,2)$, as formalized in the following set of assumptions. 

\begin{assu}\label{ass:gen}
 There exist constants
    $k\in (0,2)$ and  $K_1>0$ 
    such that 
    \begin{equation}\label{ass:K1}
    |\nabla K(x)|\le 
        K_1|x|^{k-1} 
        \qquad \text{for all} \quad x\in \R^d\setminus \{0\}.
    \end{equation}
   % for all $ x\in \R^d\setminus \{0\}$.
    %
    %\item 
    If $k \in (1,2)$,  we also require that 
    \begin{equation}\label{ass:K2}
    |\nabla K(x)|\le 
        K_2 \qquad \text{for all} \quad |x|\ge 1
    \end{equation}
  %  for all $|x|\ge 1$ and
  and
    \begin{equation}\label{ass:K3}
    |\Delta K(x)|\le 
        K_3|x|^{k-2}
         \qquad \text{for all}\quad x\in\R^d\setminus\{0\}
    \end{equation}
   % for all $\in\R^d\setminus\{0\}$,
    for some constants $K_2>0$ and $K_3>0$. 
\end{assu}

The set of conditions in Assumption \ref{ass:gen} is sufficient to obtain upper estimates of solutions.
In particular,  assumption $\eqref{ass:K2}$ is imposed to simplify the analysis and avoid working with weighted $L^p$-spaces. We require  assumption $\eqref{ass:K3}$ solely for technical simplification in the proof of the $L^p$ estimates in Lemma~$\ref{lem:estim-global-pk12}$ (see below).

The proof of our main result, Theorem~\ref{thm:main} (stated below), requires the kernel $K$ to be indeed non-smooth at the origin. This lack of regularity is expressed by the following assumptions.
% \retka{Say more clearly that the first set is sufficient for UPPER estimates? There also was a confusion between Assumptions (2.1) and 2.1 (unlucky problem with eqref and ref...Assumption 2.1/2.2.2 and Assumption (2.1)/(2.2.) are not the same thing.) Other similar confusions also exist: I corrected some. Let us recheck}

\begin{assu}\label{ass:h}
The kernel $K=K(x)$ takes the form
\begin{equation} \label{ass:K4}
	K(x) = h(|x|^k),
\end{equation}
where the constant $k \in (0,2)$ is the same as in Assumption~\ref{ass:gen},
and the function $h = h(r)$ is of class $C^2$ on
$(0, +\infty)$.
Moreover, we assume that there exist constants $c_1 > 0$ and $r_0>0$ such that
\begin{equation}\label{ass:K5}
	c_1 \le h'(r) \qquad \text{for all} \quad r \in (0,r_0]
\end{equation}
and, furthermore, if $k\in (1,2)$, there exists a constant $c_2 > 0$ such that
\begin{equation}\label{ass:K6}
    \sup_{r\in (0,r_0]}|h''(r)| \le c_2.
\end{equation}
\end{assu}

%Note that, under Assumptions \eqref{ass:K5}-\eqref{ass:K6}, the following limit exists: $\lim_{r\to 0}h'(r)>0$.

\begin{ex}\label{example}
    The following functions are our basic examples of interaction kernels satisfying all these assumptions:
    \begin{enumerate}
        \item The exponential kernel,
        \[
            K(x) = 1 - \exp(-|x|^k) \quad \text{with} \quad h(r)=1-e^{-r} \quad \text{and} \quad k\in (0,2).
        \]
        \item The power-law kernel (for higher-order singularity),
        \[
            K(x) = \frac{|x|^k}{k} \quad \text{with} \quad h(r)=\frac{r^k}{k} \quad \text{and} \quad k\in (0,1].
        \]
    \end{enumerate}
    We note that for $k \in (1,2)$, one can still consider the kernel $K(x)=|x|^k/k$ with a smooth truncation for large $|x|$, consistent with Assumptions~\eqref{ass:K2} and~\eqref{ass:K3}.
\end{ex}

Interaction kernels satisfying Assumption~\ref{ass:gen}
%\retka{Is the second set really needed for well-posedness?} 
for some $k\in (0,2)$ are  \textit{mildly singular}, in the sense that the Cauchy problem~\eqref{eq:main} admits a unique global-in-time solution for every nonnegative initial datum $u_0 \in L^1(\R^d) \cap L^\infty(\R^d)$; see Proposition~\ref{thm:K2-local}.
This solution conserves mass,
\begin{equation}\label{mass:0}
	M\equiv \int_{\R^d} u(t,x) \, \d x = \int_{\R^d} u_0(x) \, \d x  \quad \text{for all} \ t \ge 0,
\end{equation}
and preserves the nonnegativity and radial symmetry of the initial datum.
Moreover, nonnegative solutions are not only global-in-time but also uniformly bounded in time %(though not in $\varepsilon$) 
in suitable Lebesgue $L^p$-spaces; see Section~\ref{sec:git} for further properties of solutions to problem~\eqref{eq:main}.

\begin{rem} \label{rem:def-varphi}
Throughout this work,
along with the kernel $K = K(x)$ corresponding to the constant $k \in (0,2)$,
we use the function
$\varphi \in C^2\big([0,+\infty)\big)$
satisfying
\begin{equation}\label{eq:phi-prop}
	\varphi(s)=
	\begin{cases}
		s & \text{for } 0 \le s \le \left( \frac{1}{2} \right) ^{2-k},\\
		1 & \text{for } s \ge  \left( \frac{3}{2} \right) ^{2-k},
	\end{cases}
\end{equation}
with the properties
\begin{equation}\label{eq:phi-prop2}
	0 \le \varphi(s) \le \min(s,1), \quad 0 \le \varphi'(s) \le 1 \quad \text{and} \quad \varphi''(s) \le 0.
\end{equation}
In particular, we have
\begin{equation}\label{eq:phi-prop3}
\varphi(|x|^{2-k}) = |x|^{2-k} \quad \text{for}\quad  |x| \le \frac{1}{2}\quad \text{and}\quad  \varphi(|x|^{2-k}) = 1\quad \text{for}
\quad  |x| \ge \frac{3}{2}.
\end{equation}
As an example of such a~function, we may consider 
$\varphi = \psi \ast \rho(\cdot)$, where  $\psi(s) = \min(s,1)$ and $\rho$ is a suitable compactly supported mollifier.
\end{rem}

Our goal is to study the behavior of solutions to problem~\eqref{eq:main} as $\varepsilon \searrow 0$. The main result of this work, stated in the following theorem, is that the family of non-negative, radially symmetric solutions $\{u_\varepsilon\}_{\varepsilon>0}$ to problem~\eqref{eq:main}, all sharing the same initial datum, concentrates around the origin over a common finite time interval $[0,T]$.

\begin{thm}[Concentration phenomena]\label{thm:main}
Let the kernel $K$ satisfy Assumptions \ref{ass:gen} and \ref{ass:h}. % and  let $\varepsilon>0$.
Assume the initial condition $u_0 \in L^1(\R^d) \cap L^\infty(\R^d)$ to be  nonnegative and radial,
and denote by $u_\varepsilon = u_\varepsilon(t,x)$ the corresponding global-in-time nonnegative and radial solution to problem~\eqref{eq:main} with $\varepsilon>0$.
There exist numbers
    $$
    \mu\in (0,1), 
	 \quad T > 0,\quad  \nu>0,\quad  C>0, \quad \varepsilon^*>0,
    $$
depending only on
the dimension $d$,
the kernel $K$, %via Assumptions \ref{ass:gen}-\ref{ass:h},
the mass $M=\int_{\R^d} u(x,t)\, \d x$, and
the cut-off function $\varphi$ 
{\rm (}given in Remark~\ref{rem:def-varphi}{\rm )},
such that if
    \begin{equation}\label{eq:u0-concentration:intro}
        \int_{\R^d}
        \varphi ( |x|^{2-k} ) u_0(x)\,\d x
        < \mu 
        \int_{\R^d} u_0(x)\,\d x
    \end{equation}
    then 
\begin{equation}\label{main_estimate}
\int_0^T \int_{B_{(\nu \varepsilon)^{1/k}}}
u_\varepsilon(t,x)
\, \d x \, \d t\ge C \qquad\text{for all} \quad \varepsilon\in (0, \varepsilon^*).
\end{equation}
\end{thm}

Applying the H\"older inequality, we obtain the $\varepsilon$-dependent rates of concentration of the $L^p$-norms of solutions.

\begin{cor}\label{cor:lp-lower} 
Under the assumptions of Theorem~\ref{thm:main} and using the same notation,  
    for every $p\in [1,\infty)$ there exist numbers $T>0$, $\nu>0$, $C>0$, $\varepsilon^*>0$ such that 
    the solutions $u_\varepsilon$ to problem~\eqref{eq:main} satisfy
	\begin{equation}\label{eq:lp-lower-p}
		\int_0^{T}
		\left(
		\int_{B_{(\nu \varepsilon)^{1/k}}}
		u_\varepsilon (t,x)^p
		\, \d x
		\right)^\frac{1}{p}
		\, \d t
		\ge
		C%(p)
		\varepsilon^{-\frac{d(p-1)}{kp}}
	\end{equation}
	and
	\begin{equation}\label{eq:lp-lower-infty}
		\int_0^{T}
		\sup_{x \in B_{(\nu \varepsilon)^{1/k}}}
		u_\varepsilon (t,x)
		\, \d t
		\ge
		C%(p)
		\varepsilon^{-\frac{d}{k}}
	\end{equation}
    for all $\varepsilon\in (0,\varepsilon^*)$.
\end{cor}

\begin{rem}\label{rem:Linfty}
Theorem~\ref{thm:main}, together with Corollary~\ref{cor:lp-lower}, indicates that even if the interactions described by the kernel $K$ do not lead to the formation of singularities for the solution $u_\varepsilon$ for any fixed $\varepsilon > 0$ (neither in finite nor in infinite time, cf. Proposition~\ref{thm:K2-local} and the estimates of the $L^p$-norms of solutions in Lemmas \ref{lem:estim-global-p} and \ref{lem:estim-global-pk12}), 
a concentration phenomenon can still occur in $\varepsilon$-small balls for sufficiently small $\varepsilon > 0$. Indeed, from inequality~\eqref{eq:lp-lower-infty}, one can immediately deduce that
\[
    \sup_{t \in [0,T],\; |x|\le (\nu \varepsilon)^{1/k}} |u_\varepsilon(t,x)| \geq \varepsilon^{-\frac{d}{k}}C/T \quad\text{for all}\quad \varepsilon\in (0, \varepsilon^*).
\]
\end{rem}

\begin{rem}\label{rem:mu} %Theorem~\ref{thm:main} is proved in the final section of this work. 
We observe that for any non-negative initial condition $u_0 \in L^1(\mathbb{R}^d)$, the following estimate holds:
\begin{equation}
    \int_{\mathbb{R}^d} \varphi(|x|^{2-k}) u_0(x) \, \d x \le \int_{\mathbb{R}^d} u_0(x) \, \d x,
\end{equation}
which follows directly from property \eqref{eq:phi-prop} of the function $\varphi$. Consequently, assumption \eqref{eq:u0-concentration:intro} is satisfied, provided that $u_0$ is sufficiently concentrated around the origin.
\end{rem}

\begin{rem}
Theorem~\ref{thm:main} can also be used to describe the local-in-time concentration of the solution to problem~\eqref{eq:main} with fixed $\varepsilon=1$. Indeed, the proof of Theorem~\ref{thm:main} is based on the differential inequality~\eqref{eq:diff-ineq}, which for $\varepsilon=1$ takes the form 
\begin{equation}\label{eq:diff-ineq:proof}
    \frac{\d}{\d t} I(u)(t) \le \mathcal{D}(u)(t) + \omega M_u
    \big(
        - \mu M_u + I(u)(t)
    \big) \quad \text{for all } t \ge 0,
\end{equation}
involving the mass $M_u=\int_{\mathbb{R}^d} u(t,x) \, \d x$, constants $\mu \in (0,1)$, $\omega > 0$, and the quantities
\[
I(u)(t) = \int_{\mathbb{R}^d} \varphi(|x|^{2-k}) u(t, x) \, \d x,
\quad
\mathcal{D}(u)(t) = \int_{\mathbb{R}^d} \varphi(|x|^{2-k}) \Delta u(t, x) \, \d x.
\]

In particular, by choosing an initial datum of the form 
\[
u_{0,\ell}(x) = \ell v_0(x) \quad \text{with } \ell > 0,
\]
where $v_0$ is a fixed, non-negative, smooth, and compactly supported function such that 
\begin{equation}\label{eq:v0-concentration:cor}
    I(v_0) = \int_{\mathbb{R}^d} \varphi(|x|^{2-k}) v_0(x) \, \d x < \mu \int_{\mathbb{R}^d} v_0(x) \, \d x,
\end{equation}
we obtain that the corresponding solution $u_\ell = u_\ell(t,x)$ satisfies, at time $t = 0$,
\begin{equation}
\begin{split}
\left. \frac{\d}{\d t} I(u_\ell)(t) \right|_{t=0}
& \le
\mathcal{D}(u_\ell)(0) + \omega M_{u_\ell} \big(
    - \mu M_{u_\ell} + I(u_\ell)(0)
\big) \\
&= \ell \mathcal{D}(v_0) + \ell^2 \omega M_{v_0}
\big(
    - \mu M_{v_0} + I(v_0)
\big) \\
& < 0
\end{split}
\end{equation}
for sufficiently large $\ell$. Thus, by a continuity argument, there exists $T_0 > 0$ such that the solution $u_\ell(t,x)$ corresponding to the initial datum $u_{0,\ell}$ with large $\ell$ satisfies
\begin{equation}
    \frac{\d }{\d t} \int_{\mathbb{R}^d} \varphi(|x|^{2-k}) u_\ell (t, x) \, \d x < 0 \quad \text{for all } t \in [0, T_0].
\end{equation}
Since the mass of the solution is constant in time (see \eqref{mass:0}), the strict decrease of the truncated moment
$I(u_\ell)$
% $\int_{\mathbb{R}^d} \varphi(|x|^{2-k}) u_\ell(t, x) \, \d x$
over the time interval $[0, T_0]$ implies that the mass must concentrate around the origin in finite time.
\end{rem}

%%%%%%%%%%%%%%%%%%%%%%%%%%%%%%%%%%%%%%%%%%%%%%%%%%%%%%%%%%%%%

\subsection*{Notation.}
For $p\in [1,\infty]$, the norms of the Lebesgue space $L^p(\mathbb{R}^d)$  are denoted by $\|\cdot\|_p$. 
The usual Sobolev spaces are denoted by $W^{k,p}(\R^d).$
We denote by $B_r$ the ball in $\mathbb{R}^d$ centered at $x=0$ with radius $r>0$, and by $\sigma_d={2\pi^{d/2}}/{\Gamma\left(d/2\right)}$  the area of the unit sphere ${\mathbb S}^{d-1}$ in $\R^d$.
Throughout the paper, the letter  $C$ is used for various positive numbers which may vary from line to line. Its dependence upon additional parameters will be indicated explicitly.

%%%%%%%%%%%%%%%%%%%%%%%%%%%%%%%%%%%%%%%%%%%%%%%%%%%%%%%%%%%%%

\section{Well-posedness result and upper estimates}\label{sec:git}

\subsection{Existence of solutions.}
We begin by stating a standard result on well-posedness for problem~\eqref{eq:main}.

\begin{prop}\label{thm:K2-local}
	Let $d \ge 1$, $\varepsilon >0$ and $k\in (0,2)$.
    Assume that the kernel $K$ satisfies
    % assumption~\eqref{ass:K1}, and in addition, assumption~\eqref{ass:K2} if $k \in (1,2)$.
    Assumption~\ref{ass:gen}.
	For every $u_0 \in L^1(\R^d) \cap L^\infty(\R^d)$ such that $u_0\geq 0$,
	there exists
	a~unique global-in-time solution to problem~\eqref{eq:main} such that 
	 $u \in C \big( [0,\infty),L^1(\R^d)\cap L^p(\R^d)  \big)$ for each $p\in (1,\infty)$.

	This solution has the following properties.
	\begin{enumerate}
		\item  It is non-negative: $u(t,x)\geq 0$ for all $t\ge 0$ and $x\in\R^d$.
		\item It is radially symmetric for radial initial data.
		\item It exhibits the mass conservation property:
		\begin{equation}\label{mass}
		M\equiv \int_{\R^d} u(t,x) \, \d x = \int_{\R^d} u_0(x) \, \d x %\quad\text{for all}\quad t\ge 0.
		\end{equation}
        for all $t\ge 0$.
		\item It has the classical parabolic regularity:
			\begin{equation} \label{regularity}
			u \in C \big( (0,T], W^{2,p}(\R^d) \big)
			\cap C^1 \big( (0,T], L^p(\R^d) \big)
			\end{equation}
			for every $T>0$ and every $p \in (1,\infty)$. 
	\end{enumerate}
\end{prop}

The proof of Proposition \ref{thm:K2-local} is standard; analogous reasoning has been presented in other works (e.g., \cite{Carrillo-similification,LR09, LR10}).
Consequently, we will only outline the main steps of the proof.

First we state a Hardy–Littlewood–Sobolev type estimate of convolution with a~singular kernel, which is essential for both the construction of solutions to problem~\eqref{eq:main} and for the derivation of their $L^p$-estimates.

\begin{lem}\label{lem:herm2}
	Let $d \ge 1$ and $a \in [0,d)$.
    For every $p\in  \big((1-a/d)^{-1}, \infty\big]$, there exists a~constant $C_{p,a} > 0$ such that
	$$
    \left\|
    |\cdot|^{-a} \ast v
    \right\|_\infty
	\le
	C_{p,a}
	\|v\|_1^{1 - \frac{a}{\eta}}
	\|v\|_p^{\frac{a}{\eta}}
    \quad\text{with} \quad \eta= \frac{d(p-1)}{p},
	$$
	for all $v \in L^1(\R^d) \cap L^p(\R^d)$, where $\eta=d$ if $p=\infty$.
\end{lem}
Lemma~\ref{lem:herm2} is a counterpart of~\cite[Lemma 4.5.4]{H15} and can be proved in the same way.

\begin{proof}[Sketch of the proof of Proposition~\ref{thm:K2-local}]
A local-in-time solution to problem~\eqref{eq:main} is constructed by applying the Banach fixed-point theorem to the following integral equation:
\begin{equation}\label{eq:K2}
	u(t)
	=
	e^{t \varepsilon \Delta }u_{0}
	+
	\int_{0}^{t}
	\nabla e^{(t-s) \varepsilon \Delta }
	\cdot \left( u(s) \nabla K \ast u(s) \right)
	\, \d s,
\end{equation}
where the heat semigroup $e^{t \varepsilon \Delta}$ is defined by the convolution with the Gauss--Weierstrass kernel.
Here, one should use this approach in the Banach space
$C \big( [0,T],L^1(\R^d) \cap L^p(\R^d) \big)$
with a sufficiently large $p \in (1,\infty]$.

Moreover, to derive suitable estimates for the nonlinear term on the right-hand side of equation~\eqref{eq:K2} (i.e., to estimate the convolution), we proceed as follows. For $k \in (0,1)$, assumption~\eqref{ass:K1} is satisfied, allowing us to use Lemma~\ref{lem:herm2}. On the other hand, for $k \in [1,2)$, we have $\nabla K \in L^\infty(\R^d)$ by assumptions~\eqref{ass:K1} and~\eqref{ass:K2}, and the corresponding convolution is estimated using Young's inequality.

The proof that the solution is radial and non-negative is standard. Similarly, the conservation of mass is proved by integrating both sides of equation \eqref{eq:K2} with respect to $x \in \R^d$ (see \cite{Carrillo-similification, LR09, LR10}).
The parabolic regularity~\eqref{regularity} of the local-in-time solution is also classical, as well as instantaneous smoothing in terms of Sobolev norms for $t>0$, using the methods of \cite{LR09,LR10}.

% \\
% KILL
% This follows directly from the fact that $u\in C\big([0,T], L^1(\R^d)\cap L^p(\R^d)\big)$ for all $p>1$, which implies that $\nabla K \ast u\in C\big([0,T], L^p(\R^d)\big)$ for all sufficiently large $p>1$. 
% KILL
% \\
The nonnegative solutions constructed above exist for all $t\ge 0$. This is a direct consequence of a standard continuation argument combined with the $L^p$-estimates derived in the next subsection.
\end{proof}

%%%%%%%%%%%%%%%%%%%%%%%%%%%%%%%%%%%%%%%%%%%%%%%%%%%%%%%%%%%%%%%%%%%%%%%%

\subsection{\texorpdfstring{$L^p$}--estimates.}\label{subsec:Lp}

In this subsection, we derive the upper, $\varepsilon$-dependent, $L^p$-estimates of solutions to problem~\eqref{eq:main} which are obtained independently for the cases $k \in (0,1]$ and $k \in (1,2)$. In this way, we extend the result from the works \cite{BBKL21,BBKL22}, where the case $k = 1$ was considered, to cover all exponents $k \in (0,2)$.

\begin{rem} \label{rem:scale}
The following two properties of solutions to problem~\eqref{eq:main} are essential for the proof of the $L^p$-decay estimates.
\begin{enumerate} 
\item The solutions to problem~\eqref{eq:main} obtained in Proposition~\ref{thm:K2-local} possess the regularity stated in \eqref{regularity}. This property allows us to perform integration by parts and time differentiations of integrals involving $u_\varepsilon$ in the subsequent calculations.
\item
	If a function $u$ solves problem~\eqref{eq:main} with $\varepsilon = 1$ and initial condition $u_0$, then the function $u_\varepsilon(t,x) = \varepsilon u(\varepsilon t,x)$ is a solution to problem~\eqref{eq:main} with an arbitrary $\varepsilon > 0$ and initial condition $\varepsilon u_0$.
	Consequently, in the proofs of the following lemmas, we first set $\varepsilon = 1$ for clarity, and subsequently rescale the result accordingly.
\end{enumerate}
\end{rem}

\begin{rem}[Mass dependence via scaling]
Here, we consider  the equation 
\begin{equation}\label{eq:homogeneous}
u_t = \Delta u + \nabla \cdot (u(\nabla K * u)) \quad \text{with a kernel satisfying}\quad  K(\lambda x) = \lambda^k K(x).
\end{equation} 
Suppose now that $\tilde{u}(t,x)$ is a solution of this equation with unit mass, $\int_{\mathbb{R}^d} \tilde{u}(t,x) \, \d x = 1$, and such  that its $L^p$-norm is uniformly bounded: $\|\tilde{u}(t, \cdot)\|_p \le C_1(p)$ for all $t > 0$.
Due to the scaling property of equation \eqref{eq:homogeneous}, 
the rescaled function
$$\tilde{u}_\lambda(t,x) = \lambda^{d+k} \tilde{u}(\lambda^2 t, \lambda x)$$ is also its solution for each $\lambda>0$. 
 The mass of the scaled solution is exactly $M = \lambda^k$. Using the scaling parameter $\lambda = M^{\frac{1}{k}}$, we obtain 
 that the 
  new solution $u(t,x) \equiv \tilde{u}_\lambda(t,x)$ with arbitrary mass $M$ satisfies 
$$ \|u(t, \cdot)\|_p \le M^{1+\frac{\eta}{k}} C_1(p), $$
with  $\eta = {d(p-1)}/{p}$ and 
with a constant $C_1(p)$ independent of the mass $M$. 
Analogously, if the gradient of the unit-mass solution is bounded: $\|\nabla \tilde{u}(t, \cdot)\|_p \le C_2(p)$ for all $t > 0$, then by applying the same scaling, we obtain the estimate for the gradient of the solution with mass $M$:
$$ \|\nabla u(t, \cdot)\|_p \le M^{1+\frac{1 + \eta}{k}} C_2(p). $$
This remark explains the powers of the mass $M$ appearing below, which are obtained even for non-homogeneous kernels. 
Meanwhile, the powers of $\varepsilon$ result from another scaling introduced in Remark \ref{rem:scale}.
\end{rem}
%\retka{GN included here. I stupidly copied it from BBKL2. Citation in the commentary below, I cannot access the bib file}

In the following, we systematically use the following classical estimates of the Sobolev norms.

\begin{lem}[Gagliardo-Nirenberg-Sobolev inequality  {\cite{BM19}}]  \label{GN} 
For every  $v\in C_c^{\infty}(\R^d)$, the following inequality holds 
$$
\|v\|_{\dot{W}^{\beta,r}} \leq C \|v\|^{\theta}_{\dot{W}^{m,p}} \|v\|^{1-\theta}_{q},
$$
where $m>\beta\geq 0$, and $r$ is defined by
$$
\frac{d}{r}=\beta-\theta \Big( m-\frac{d}{p} \Big)+(1-\theta)\frac{d}{q},
$$
under the assumption that $\beta/m \leq \theta < 1$, except  when $\beta=0$, $r=q=\infty$ and $m-d/p$ is a nonnegative integer. Here, the constant $C$ depends on $m,p,q,\beta,d$.
\end{lem}

\begin{lem}\label{lem:estim-global-p} 
Let $d \ge 1$, $\varepsilon>0$
and $k \in (0,1]$,
and assume that the kernel $K$ satisfies assumption~\eqref{ass:K1}.
Fix a non-negative initial condition $u_0 \in L^1(\R^d) \cap L^\infty(\R^d)$.
Let  $u_\varepsilon = u_\varepsilon(t,x)$ be a non-negative solution to problem~\eqref{eq:main},
defined for all $t \ge 0$. % with  $T_{max}\in (0,\infty]$.
For every $p \in (1,\infty)$, there exists a number $C > 0$
\textup{(}independent of the solution $u_\varepsilon$ and parameter $\varepsilon$\textup{)}
such that 
\begin{equation}\label{eq:estim-p}
		\| u_\varepsilon(t) \|_p \le \max
        \left\{
        M,
		\| u_0 \|_{\max{\{\bar{p},p\}}}, 
		 C M^{1 + \eta/k} \varepsilon^{-\eta/k}
		\right\},
\end{equation}
for all $t \ge 0$. In inequality \eqref{eq:estim-p}, we denote
\begin{equation}\label{eq:def-pk}
    \eta = \frac{d(p-1)}{p}
\end{equation}
and $\bar{p} \in [2,\infty)$ is a fixed exponent satisfying $\bar{p}>1/(1-(1-k)/d).$ 
\end{lem}

\begin{proof}
In this proof, we use ideas from the paper~\cite{BBKL21} (case $k=1$); however, the argument naturally extends to $k \in (0,1)$.

First, we consider 
\begin{equation}\label{p:ass}
p \in [2,\infty) \quad\text{such that} \quad p>\frac{1}{1-\frac{1-k}{d}}.
\end{equation}
The inequality $p > 1/(1 - (1 - k)/d)$ in~\eqref{p:ass} plays a role only if $d = 1$ and $k \in (0, 1/2]$; however, this observation does not affect the subsequent reasoning.

As mentioned in Remark~\ref{rem:scale}, we begin by considering problem~\eqref{eq:main} with $\varepsilon = 1$, and we denote $u=u_1$.
Multiplying the first equation in~\eqref{eq:main} by $(p-1)^{-1} u^{p-1}$ and integrating by parts, we obtain
\begin{equation}\label{eq:ineqT2}
	\begin{split}
		\frac{1}{p(p-1)}  \frac{\d}{\d t}  \| u \|_p^p  = -&
		%\varepsilon
		\int_{\R^d} |\nabla u|^2 u^{p-2} \, \d x - \int_{\R^d} u^{p-1} \nabla u \cdot (\nabla K \ast u) \, \d x \\
		= - &
		%\varepsilon 
		\frac{4}{p^2} \|\nabla u^{p/2} \|_2^2
		- \frac{2}{p} \int_{\R^d} u^{p/2} \nabla u^{p/2} \cdot
        (\nabla K \ast u  
         ) \, \d x.
        \end{split}
\end{equation}
Let us recall that the kernel $K$ satisfies assumption~\eqref{ass:K1}. Thus, by  the H{\"o}lder inequality combined with Lemma~\ref{lem:herm2} for $a=1-k$ (using \eqref{p:ass}), we get:
\begin{equation}\label{eq:ineqT1}
- \frac{2}{p} \int_{\R^d} u^{p/2} \nabla u^{p/2} \cdot \big( \nabla K \ast u \big)\, \d x
\le
C
    M^{1 - (1-k)/\eta}
    \|u\|_p^{p/2 + (1-k)/\eta}
  \|\nabla u^{p/2} \|_2,
\end{equation}
where $\eta$ is defined in formula~\eqref{eq:def-pk}. 
Applying this inequality to 
equation ~\eqref{eq:ineqT2}, we obtain
\begin{equation}\label{eq:ineqT5}
    \begin{split}
		\frac{1}{p(p-1)}  \frac{\d}{\d t}  \| u \|_p^p 
        \le &
		\frac{4}{p^2} \|\nabla u^{p/2} \|_2
        \Big(
        - \|\nabla u^{p/2} \|_2
        + C
        M^{1 - (1-k)/\eta}
        \|u\|_p^{p/2 + (1-k)/\eta}
        \Big).
    \end{split}
\end{equation}

Next, using the Gagliardo-Nirenberg-Sobolev inequality and then the H{\"o}lder inequality, we deduce
\begin{align*}
		\| u\|_p^{p/2}
         = \| u^{p/2} \|_2 & \le C \| \nabla u^{p/2} \|_2^{d/(d+2)} \| u^{p/2} \|_1^{2/(d+2)} \\ & = C \| \nabla u^{p/2} \|_2^{d/(d+2)} \| u \|_{p/2}^{p/(d+2)} \\
         & 
        \le C \| \nabla u^{p/2} \|_2^{d/(d+2)} \|u\|_p^{\beta p / (d+2)} M^{(1 - \beta) p / (d+2)},
\end{align*}
with $\beta=(p-2)/(p-1)$
Therefore,
\begin{equation*}
    \| \nabla u^{p/2} \|_2^{d/(d+2)}
    \ge
    C \|u\|_p^{p (d+2-2 \beta) / 2(d+2)} M^{-(1 - \beta) p / (d+2)},
\end{equation*}
which implies that
\begin{equation}\label{eq:ineqp3}
    \| \nabla u^{p/2} \|_2
    \ge
    C M^{-1/\eta}
    \|u\|_p^{p/2 + 1/\eta}.
\end{equation}
Combining estimate~\eqref{eq:ineqp3} with inequality~\eqref{eq:ineqT5}, we obtain the differential inequality
\begin{equation}\label{eq:ineqpn1}
		\frac{1}{p(p-1)}  \frac{\d}{\d t}  \| u \|_p^p 
        \le
		C_1
        M^{1 - (1-k)/\eta}
        \| \nabla u^{p/2}\|_2
        \|u\|_p^{p/2 + (1-k)/\eta}
        \left(
        1 - C_2 M^{- 1 - k/\eta } \|u\|_p^{k/\eta}
        \right),
\end{equation}
where $C_1, C_2 > 0$ are constants depending only on $p, d, k,$ and $K$.
By applying the reasoning from the proof of~\cite[Lemma~4.1]{BBKL21} to this differential inequality, we deduce that
\begin{equation}\label{eq:ineqp7}
	\| u(t) \|_p \le \max
    \left\{ 
	\| u_0 \|_p, C M^{1+\eta/k}
    \right\}
	%\quad \text{for all} \quad t > 0
\end{equation}
for all $t \ge 0$,
which holds true for all exponents $p$ satisfying condition~\eqref{p:ass}
with a constant $C=C(p,d,k,K) > 0$.

Now, we fix an arbitrary exponent $\bar{p}$ satisfying the condition~\eqref{p:ass}
and prove inequality~\eqref{eq:estim-p} for all $p\in (1,\bar{p})$.
Using the H\"older inequality and estimate~\eqref{eq:ineqp7} applied to $p = \bar{p}$ with $\bar{\eta} = d(\bar{p} - 1)/\bar{p}$,
we have
\begin{equation}\label{eq:ineqp8}
	\begin{split}
		\|u(t)\|_p
		& \le M^\frac{\bar{p} - p}{p(\bar{p} - 1)}
        \|u(t)\|_{\bar{p}}^\frac{\bar{p} (p  - 1)}{p(\bar{p} - 1)} \\
		& \le M^\frac{\bar{p} - p}{p(\bar{p} - 1)}
		\left(
            \max
            \left\{ 
        	\| u_0 \|_{\bar{p}}, C M^{1+\bar{\eta}/k}
            \right\}
        \right)^\frac{\bar{p} (p  - 1)}{p(\bar{p} - 1)} \\
		& \le \max \left\{ M^\frac{\bar{p} - p}{p(\bar{p} - 1)} \| u_0 \|_{\bar{p}}^\frac{\bar{p} (p  - 1)}{p(\bar{p} - 1)},
		C M^{1+\eta/k}
		\right\} \\
		& \le \max
		\left\{ M, \| u_0 \|_{\bar{p}},
		C M^{1+\eta/k}
		\right\}.
	\end{split}
\end{equation}

We have already proved estimate~\eqref{eq:estim-p} for $\varepsilon=1$ and for all $p\in (1,\infty)$. To extend this result to all $\varepsilon>0$, it suffices to use property (2) from
Remark~\ref{rem:scale} and the  substitution
$u(t,x)={\varepsilon^{-1}}u_\varepsilon \left(\varepsilon^{-1} t,x\right)$.
\end{proof}

\begin{lem}\label{lem:estim-global-pk12}
Let $d \ge 1$, $\varepsilon>0$
and $k \in (1,2)$,
and assume that the kernel $K$ satisfies assumption~\eqref{ass:K3}.
Fix a non-negative initial condition $u_0 \in L^1(\R^d) \cap L^\infty(\R^d)$.
Let  $u_\varepsilon = u_\varepsilon(t,x)$ be a non-negative solution to problem~\eqref{eq:main},
defined for all $t \ge 0$.
For every $p \in (1,\infty)$, there exists a number $C > 0$
\textup{(}independent of the solution $u_\varepsilon$ and
parameter $\varepsilon$\textup{)}
such that 
    \begin{equation}\label{eq:estim-pk12}
		\| u_\varepsilon(t) \|_p \le \max
        \left\{
        M,
		\| u_0 \|_{\max{\{\bar{p},p\}}}, 
		  C M^{1 + \eta/k} \varepsilon^{-\eta/k}
		\right\},
	\end{equation}
for all $t \ge 0$.
Here, the constant $\eta$ is defined in formula~\eqref{eq:def-pk},
and
$\bar{p} \in [2,\infty)$ is a fixed exponent that satisfies $\bar{p}>1/(1-(2-k)/d).$ 
\end{lem}

\begin{proof}
First, we consider an exponent
\begin{equation}\label{p:ass2}
p \in [2,\infty)
\quad\text{satisfying}
\quad p> \left( 1-\frac{2-k}{d} \right)^{-1}.
\end{equation}
In fact, the inequality
%$p > 1/(1 - (2 - k)/d)$
in~\eqref{p:ass2} imposes a restriction only if $d = 1$ and $k \in (1, 3/2]$; 
%\retka{3/2 should be included?}; 
however, this observation does not affect the subsequent reasoning.

We proceed as in the proof of Lemma~\ref{lem:estim-global-p}.
Assume that $\varepsilon = 1$. We multiply the first equation in~\eqref{eq:main} by $(p-1)^{-1} u^{p-1}$ and integrate over $\R^d$. Applying integration by parts (which is applied twice to the drift term) and H{\"o}lder's inequality, we obtain
\begin{equation}\label{eq:ineq-B1}
    \begin{split}
        \frac{1}{p(p-1)}  \frac{\d}{\d t}  \| u \|_p^p  & = 
        - \int_{\R^d} |\nabla u|^2 u^{p-2} \, \d x 
        + \frac{1}{p} \int_{\R^d} u^p \left(\Delta K \ast u\right) \, \d x \\
        & \le
        - \frac{4}{p^2} \| \nabla u^{p/2} \|_2^2
        + \frac{1}{p} \| u \|_p^p \, \| \Delta K \ast u \|_\infty.
    \end{split}
\end{equation}

Recalling that the kernel $K$ satisfies assumption~\eqref{ass:K3}, we apply Lemma~\ref{lem:herm2} with the parameter $a = 2-k$. The condition $p > (1 - (2-k)/d)^{-1}$ is satisfied due to \eqref{p:ass2}. We~obtain
\begin{equation}\label{eq:ineq-B2}
    \| \Delta K \ast u \|_\infty \le C_{p,k} \|u\|_1^{1 - (2-k)/\eta} \|u\|_p^{(2-k)/\eta},
\end{equation}
where the number $\eta$ was defined in formula~\eqref{eq:def-pk}.
Substituting this estimate into~\eqref{eq:ineq-B1} and using the mass conservation~\eqref{mass}, we have
\begin{equation}\label{eq:ineq-B3}
     \frac{1}{p(p-1)} \frac{\d}{\d t} \| u \|_p^p \le - \frac{4}{p^2} \| \nabla u^{p/2} \|_2^2 + C M^{1 - (2-k)/\eta} \|u\|_p^{p + (2-k)/\eta}.
\end{equation}
Applying the Gagliardo--Nirenberg--Sobolev estimate~\eqref{eq:ineqp3} to the term $\| \nabla u^{p/2} \|_2^2$ and combining it with inequality~\eqref{eq:ineq-B3}, we get
\begin{equation}\label{eq:ineqpn2-df}
	\begin{split}
		\frac{1}{p(p-1)} \frac{\d}{\d t} \| u \|_p^p \le & 
		C_1 M^{1 - (2-k)/\eta}
        \lVert u \rVert_p^{p + (2-k)/\eta}
        \left(
        1 - C_2 M^{-1 - k/\eta} \| u \|_p^{k/\eta}
        \right),
	\end{split}
\end{equation}
with some numbers $C_1, C_2 > 0$.

Once again, reasoning as in the proof of~\cite[Lemma~4.1]{BBKL21}, one can show that
\begin{equation*}
	\| u(t) \|_p \le \max \left\{ 
	\| u_0 \|_p, C M^{1+\eta/k} \right\}
	% \quad \text{for all} \quad t > 0,
\end{equation*}
for all $t \ge 0$, where the number $ C > 0$ is independent of $M$. 
The proof for $p \in (1,\bar{p})$ is analogous to the calculations in inequality~\eqref{eq:ineqp8}. The rescaling in terms of $\varepsilon$ follows from Remark~\ref{rem:scale}.
\end{proof}

In the following, while studying the case of $d=1$ and $k \in (1,2)$, we also require $\varepsilon$-dependent estimates of $\partial_x u_\varepsilon$, which we prove in the following lemma.

\begin{lem}\label{lem:estim-global-p-d1-k12}
Let $d = 1$, $\varepsilon>0$,
and $k \in [1,2)$,
and assume that the kernel $K$ satisfies assumption~\eqref{ass:K3}.
Fix a non-negative initial condition $u_0 \in W^{1,1}(\R) \cap W^{1,\infty}(\R)$.
Let  $u_\varepsilon = u_\varepsilon(t,x)$ be a non-negative solution to problem~\eqref{eq:main},
defined for all $t \ge 0$.
For every $p \in [2,\infty)$, there exists a number $C > 0$
{\rm (}independent of the solution $u_\varepsilon$ and of the
parameter $\varepsilon${\rm )} 
such that 
\begin{equation}\label{eq:estim:partial}
\|\partial_{x}u_{\varepsilon}(t)\|_{p}\le 
\max
        \left\{
		\| (u_0)_x \|_{p},  
		  C M^{1 + (1+\eta)/k} \varepsilon^{-(1+\eta)/k}
		\right\},
\end{equation}
for all $t \ge 0$.
In inequality~\eqref{eq:estim:partial}, the number $\eta$ is defined as in Lemma~\ref{lem:estim-global-pk12}.
\end{lem}

\begin{proof}
Applying the same reasoning as in the previous lemmas, we begin by setting $\varepsilon=1$ and denoting $u(t,x)=u_1(t,x)$.

To establish this proof, we need the inequality
\begin{equation} \label{uxp-A}
\frac{\d}{\d t} \Vert |u_x|^{p/2} \Vert_2^2
\le
-C \Vert (|u_x|^{p/2})_x \Vert_2^2
+
C M^{1 + kp/(2p + 1)}
\Vert (|u_x|^{p/2})_x \Vert_2^{2-2k/(2p+1)}
\end{equation}
and the estimate 
\begin{equation} \label{uxp-B}
\Vert |u_x|^{p/2} \Vert_2 \le C M^{p/(2p+1)} \Vert (|u_x|^{p/2})_x \Vert_2^{(2p-1)/(2p+1)}.
\end{equation}
We will also need the following crucial estimate:
\begin{equation} \label{uxp-C1}
\Vert u \Vert_{\infty} \le C M^{(p+1)/(2p+1)} \Vert (|u_x|^{p/2})_x \Vert_2^{2/(2p+1)}
\end{equation}
and its corollary, obtained by interpolating between $M$ and $\Vert u \Vert_{\infty}$:
\begin{equation} \label{uxp-C2}
\Vert u \Vert_{2} \le C M^{(3p+2)/(4p+2)} \Vert (|u_x|^{p/2})_x \Vert_2^{1/(2p+1)}.
\end{equation}

We begin by proving inequality~\eqref{uxp-B}. We integrate by parts, and then we use the Cauchy-Schwarz and the H{\"o}lder inequalities:
\begin{equation*}
	\begin{split}
\Vert |u_x|^{p/2} \Vert_2^2
    & = - (p-1) \int_\R{u |u_x|^{p-2} u_{xx}}\,\d x
    \\
    & \le (p-1)
    \left(
    \int_\R{u |u_x|^{p-2}} u_{xx}^2\, \d x
    \right)^{1/2} 
    \left(
    \int_\R{u |u_x|^{p-2}} \, \d x
    \right)^{1/2}
    \\
    & \le  (p-1)
    \left(
    \frac{2}{p}
    \Vert u \Vert_{\infty}^{1/2} \Vert (|u_x|^{p/2})_x \Vert_2 
    \right) 
    \left(
    M^{1/2} \Vert u_x \Vert_{\infty}^{(p-2)/2}
    \right).
	\end{split}
\end{equation*}
Now we apply the Gagliardo-Nirenberg-Sobolev inequality, first to $u$ and then to $(u_x)^{p/2}$:
\begin{equation*}
	\begin{split}
    \Vert |u_x|^{p/2} \Vert_2^2  & \le C \big(M^{1/4} \Vert u_x \Vert_{\infty}^{1/4} \Vert (|u_x|^{p/2})_x \Vert_2 \big)  \big(M^{1/2} \Vert u_x \Vert_{\infty}^{(p-2)/2} \big)
    \\
    & = C M^{3/4}  \Vert |u_x|^{p/2} \Vert_{\infty}^{(2p-3)/2p} \Vert (|u_x|^{p/2})_x \Vert_2
    \\
    & \le C M^{3/4}  \Vert |u_x|^{p/2} \Vert_{2}^{(2p-3)/4p} \Vert (|u_x|^{p/2})_x \Vert_2^{(6p-3)/4p}.
	\end{split}
\end{equation*}
Therefore, we obtain
\begin{equation*}
\Vert |u_x|^{p/2} \Vert_2^{(6p+3)/4p} \le C M^{3/4} \Vert (|u_x|^{p/2})_x \Vert_2^{(6p-3)/4p},
\end{equation*}
which implies inequality~\eqref{uxp-B}.

Now we prove inequality~\eqref{uxp-C1} by using the Gagliardo-Nirenberg-Sobolev inequality for $u$
and inequality \eqref{uxp-B}:
\begin{equation*}
	\begin{split}
\Vert u \Vert_{\infty}  & \le C M^{(p-1)/(2p-1)} \Vert u_x \Vert_p^{p/(2p-1)}
    \\
    & = C M^{(p-1)/(2p-1)} \Vert |u_x|^{p/2} \Vert_2^{2/(2p-1)}
    \\
    & \le C M^{(p-1)/(2p-1)} M^{2p/((2p+1)(2p-1))} \Vert (|u_x|^{p/2})_x \Vert_2^{2/(2p+1)}
    \\
    & = C M^{(p+1)/(2p+1)} \Vert (|u_x|^{p/2})_x \Vert_2^{2/(2p+1)}.
	\end{split}
\end{equation*}
Now we are ready to prove \eqref{uxp-A}. 
Using the equation for $u_x$ and integrating by parts we get:
\begin{equation}\label{eq:uxp-df}
	\begin{split}
		\frac{1}{p(p-1)}  \frac{\d}{\d t}  \| |u_x|^{p/2} \|_2^2  \le & 
-\frac{4}{p^2} \Vert (|u_x|^{p/2})_x \Vert_2^{2}+\int_\R{|u_x|^{p-2} u_{xx} 
(u (K' \ast u))_x} \,\d x. 
	\end{split}
\end{equation}
Denoting the second term on the right-hand side by $N$ and integrating by parts again, we obtain
\begin{equation*}
	\begin{split}
    N&= \int_\R{|u_x|^{p-2} u_x u_{xx} (K' \ast u)}\, \d x +\int_\R{u |u_x|^{p-2} u_{xx} (K'' \ast u)}\, \d x
    \\
    &= -\frac{1}{p} \int_\R {|u_x|^{p} (K'' \ast u)} \, \d x+\int_\R{u |u_x|^{p-2} u_{xx} (K'' \ast u)}\, \d x.
	\end{split}
\end{equation*}
Therefore, applying the H{\"o}lder inequality twice, then Lemma~\ref{lem:herm2} for $p=\infty$ and finally the Gagliardo-Nirenberg-Sobolev inequality applied to $(u_x)^{p/2}$, we get:
\begin{equation*}
	\begin{split}
    |N| & \le C \left( \Vert |u_x|^{p/2} \Vert_2^2
    +
    \int_\R{ u |u_x|^{p-2} |u_{xx}| }\, \d x
    \right) \Vert K'' \ast u \Vert_{\infty}
    \\
    & \le C \left( \Vert |u_x|^{p/2}  \Vert_2^2 + \left(\int{u^2\,\d x}\right)^{1/2} \left(\int{u_{xx}^2|u_x|^{p-2}} \, \d x\right)^{1/2} \Vert u_{x} \Vert_{\infty}^{(p-2)/2}
    \right) \Vert K'' \ast u \Vert_{\infty}
    \\
    & \le C \left( \Vert |u_x|^{p/2} \Vert_2^2 +
    \frac{2}{p} \Vert u \Vert_2 \Vert (|u_x|^{p/2})_x \Vert_2 \Vert |u_x|^{p/2} \Vert_{\infty}^{(p-2)/p}
    \right) \Vert u \Vert_1^{k-1} \Vert u \Vert_{\infty}^{2-k}
    \\
    & \le C\left( \Vert |u_x|^{p/2} \Vert_2^2 + \Vert u \Vert_2 \Vert (|u_x|^{p/2})_x \Vert_2^{(3p-2)/2p} \Vert |u_x|^{p/2} \Vert_{2}^{(p-2)/2p}
    \right) M^{k-1} \Vert u \Vert_{\infty}^{2-k}.
	\end{split}
\end{equation*}
Now, inequality \eqref{uxp-A} follows from \eqref{uxp-B}, \eqref{uxp-C1} and \eqref{uxp-C2}.

To conclude (as in the two preceding lemmas), it remains to observe that inequality~\eqref{uxp-B} is equivalent to the following one
$$
\Vert (|u_x|^{p/2})_x \Vert_2 \ge C M^{-p/(2p-1)} 
\Vert |u_x|^{p/2} \Vert_2^{(2p+1)/(2p-1)}.
$$
Therefore, by \eqref{uxp-A} we obtain 
\begin{align*} 
\frac{\d}{\d t} \Vert |u_x|^{p/2} \Vert_2^2
& \le C \Vert (|u_x|^{p/2})_x \Vert_2^2
\left(-1+C M^{1+kp/(2p+1)} \Vert (|u_x|^{p/2})_x \Vert_2^{-2k/(2p+1)}\right)
\\
& \le C \Vert (|u_x|^{p/2})_x \Vert_2^2
\left(-1+C M^{1+kp/(2p+1)} M^{2kp/(4p^2-1)} \Vert |u_x|^{p/2} \Vert_2^{-2k/(2p-1)}\right)
\\
& \le C \Vert (|u_x|^{p/2})_x \Vert_2^2
\Vert |u_x|^{p/2} \Vert_2^{-2k/(2p-1)} \\
& \quad \times \left(-(\Vert |u_x|^{p/2} \Vert_2^{2/p})^{kp/(2p-1)}+C M^{1+kp/(2p-1)}\right)
\\
& \le C \Vert (|u_x|^{p/2})_x \Vert_2^2
\Vert |u_x|^{p/2} \Vert_2^{-2k/(2p-1)}
\left(
-\Vert u_x \Vert_p^{kp/(2p-1)}
+C M^{1+kp/(2p-1)}
\right).
\end{align*}
Consequently, since
$$
\frac{2p-1}{kp} \left( 1+\frac{kp}{2p-1} \right)=1+\frac{2p-1}{kp}
=1+\frac{1+\eta}{k},
$$
(recalling that here $\eta=(p-1)/p$ since $d=1$), as in the two preceding lemmas we obtain
$$
\|\partial_{x}u(t)\|_{p}\le 
\max
        \left\{
		\| (u_0)_x \|_{p},  
		  C M^{1 + (1+\eta)/k}
		\right\},
$$
and it only remains to rescale with respect to $\varepsilon$ to obtain the general case.
\end{proof}

%%%%%%%%%%%%%%%%%%%%%%%%%%%%%%%%%%%%%%%%%%%%%%%%%%%%%%%%%%%%%

\section{Concentration lemma}\label{sec:cos}

We describe the concentration
of a radial solution $u_\varepsilon$ to problem~\eqref{eq:main}  
at the origin
by studying the following quantity
\begin{equation}\label{eq:def-D:n}
\D(u_\varepsilon)(t)  =  \int_{\R^d}
\varphi
( |x|^{2-k} ) 
\Delta u_\varepsilon(t,x) \, \d x,
\end{equation}
where the function $\varphi$ %(|x|^{2-k})$
is defined in Remark~\ref{rem:def-varphi} and the parameter $k\in (0,2)$ is determined by the integration kernel $K=K(x)$ via Assumptions~\ref{ass:gen} and~\ref{ass:h}.
The following result is crucial for proving Theorem~\ref{thm:main} and it 
states that $\D(u_\varepsilon)$ %after integration in time 
grows at least as $\varepsilon^{-1}$ when $\varepsilon \searrow 0$.

\begin{lem}[Concentration lemma]\label{lem:conc}
Let $u_\varepsilon = u_\varepsilon(t,x)$ be a radial, non-negative, global-in-time solution  to problem~\eqref{eq:main} with $\varepsilon > 0$, corresponding to a radial, non-negative initial condition $u_0 \in L^1(\R^d) \cap L^\infty(\R^d)$.
There exist numbers
    $$
    \mu\in (0,1), \quad
	\omega > 0, \quad 
	T_0 > 0
    $$
\textup{(}see definitions~\eqref{eq:def-mu-omega} and~\eqref{eq:def-T0}, below\textup{)},    
depending only on
the dimension $d$,
the kernel $K$, the exponent $k$,
the mass $M=\int_{\R^d} u(x,t)\, \d x$, and
the cut-off function $\varphi$ 
{\rm (}given in Remark~\ref{rem:def-varphi}{\rm )},
such that if
    \begin{equation}\label{eq:u0-concentration}
        \int_{\R^d}
        \varphi ( |x|^{2-k} ) u_0(x)\,\d x
        < \mu M
        % \int_{\R^d} u_0(x)\,\d x
    \end{equation}
    then for every $T > T_0$ there exists $L = L(T) > 0$
    %\retka{$L$ also depends on the same parameters as $\mu$ etc.}
    {\rm (}see definition~\eqref{eq:def-L}, below{\rm )},
    such that
\begin{equation}\label{eq:estim-D}
    \varepsilon\int_0^{T}
    \D(u_\varepsilon)(t)
    e^{-\omega M t}
    \, \d t
    \ge
    L
\end{equation}   
for all $\varepsilon>0$.
\end{lem}
\begin{proof}
To simplify the notation in this proof, we write $u$ instead of $u_\varepsilon$. 
Note that a key difference from the proofs in Section~\ref{sec:git} is that here we 
consider all $\varepsilon > 0$; we cannot simply choose $\varepsilon = 1$ and 
proceed by a rescaling argument.
%\retka{rather "here there is no simple rescaling argument"?}

We introduce the truncated moment of order $2-k$
	\begin{equation}\label{ew:def-It}
		I(t) \equiv \int_{\R^d} \varphi(|x|^{2-k}) u(t,x) \, \d x, %\quad \text{for all } t \ge 0,
	\end{equation}
where the function $\varphi$ is defined in Remark~\ref{rem:def-varphi}.

Our  goal is to derive a differential inequality for $I(t)$ (see~inequality \eqref{eq:diff-ineq}, below).
We multiply equation~\eqref{eq:main} by
$\varphi ( |x|^{2-k} )$
and integrate the resulting identity with respect to $x\in \R^d$ to obtain the relation
    \begin{equation}\label{eq:moment}
        \frac{\d}{\d t} I(t)= \varepsilon \D(u)(t)+J(t),
    \end{equation}
where
\begin{equation*}
    J(t) \equiv \int_{\R^d}
    \varphi ( |x|^{2-k} )
    \nabla \cdot 
    \left(u(t,x) \nabla K \ast u(t,x) \right)
    \, \d x. 
\end{equation*}

Integrating by parts, using the radial symmetry of the kernel $K$ and the definition of the weight 
$\varphi$, and applying a symmetrization argument, we obtain
\begin{equation*}
    \begin{split}
        J(t)
        = & - (2-k) \int_{\R^d} u(t,x) \varphi'
        ( |x|^{2-k} ) \frac{x}{|x|^k} \cdot \nabla K \ast u(t,x) \, \d x \\
        = & - k(2-k) \int_{\R^d} \int_{\R^d}
        u(t,x) u(t,y)
        \varphi' ( |x|^{2-k} )
        \frac{x}{|x|^k}
        \cdot
        \frac{x-y}{|x-y|^{2-k}}
        h' ( |x-y|^k )
        \, \d y \, \d x \\
        = & - \frac{k(2-k)}{2}
        \int_{\R^d} \int_{\R^d}
        u(t,x) u(t,y) 
        \Phi(x,y) \, \d x \, \d y,
    \end{split}
\end{equation*}
where
\begin{equation}\label{eq:def-Phi}
    \Phi(x,y) \equiv
    \left(
        \varphi' ( |x|^{2-k} ) \frac{x}{|x|^k} -
        \varphi' ( |y|^{2-k} ) \frac{y}{|y|^k}
    \right)
    \cdot
    \frac{x-y}{|x-y|^{2-k}}
    h' ( |x-y|^k ).
\end{equation}
Since $u$ is radially symmetric, using the substitution $y \mapsto -y$, we obtain
\begin{equation}\label{eq:J}
    J(t) = - \frac{k(2-k)}{4} \int_{\R^d} \int_{\R^d}
    u(t,x) u(t,y) \big( \Phi(x,y) + \Phi(x,-y) \big) \, \d x \, \d y.
\end{equation}

We introduce the function
\begin{equation}\label{eq:def-psi}
    \Psi(x,y) \equiv  \Phi(x,y) + \Phi(x,-y),
\end{equation}
which has properties studied in Appendix~\ref{appendix}.
In particular, by Lemmas \ref{lem:Psi:lowerbound} and \ref{lem:Psi:upperbound},
there exist numbers $r_1\in(0,1/2]$, $\psi_1, \psi_2 > 0$ such that
\begin{equation}\label{eq:phi-lower}
	\Psi(x,y) \ge \psi_1
    \quad
    \text{for all}
    \quad
    (x,y) \in B_{r_1} \times B_{r_1}
\end{equation}
and
\begin{equation}\label{eq:phi-upper}
	|\Psi(x,y)| \le \psi_2 
    \quad
    \text{for all}
    \quad
    (x,y) \in
    \left(
    \R^d \times \R^d
    \right)
    \setminus 
    \left(
    B_{r_1} \times B_{r_1}
    \right)
    .
\end{equation}

%\retka{Maybe only introduce the upper bound later when we use it?}
Consequently, by inequality \eqref{eq:phi-lower}, the quantity
	$$
	J_1(t) \equiv
	 - \frac{k(2-k)}{4} \int_{B_{r_1}} \int_{B_{r_1}}
		u(t,x) u(t,y) \Psi(x,y) \, \d x \, \d y,
	$$
can be estimated from above:
	$$
	J_1(t) \le
	 - \frac{k(2-k) \psi_1}{4} \int_{B_{r_1}} \int_{B_{r_1}}
		u(t,x) u(t,y) \, \d x \, \d y.
	$$
By conservation of mass~\eqref{mass:0}, the inclusion
\begin{equation}\label{eq:R-incl}
    (\R^d \times \R^d) \setminus (B_{r_1} \times B_{r_1})
    \subset
    \big( \R^d \times (\R^d \setminus B_{r_1}) \big) \cup \big( (\R^d \setminus B_{r_1}) \times \R^d \big),
\end{equation}
the symmetry of $u(\cdot,t)$ and the inequality (cf. properties \eqref{eq:phi-prop}-\eqref{eq:phi-prop3})
\begin{equation}\label{eq:phi-estim}
	r_1^{2-k}
	\le
	\varphi(|x|^{2-k})
	\quad
	\text{for all }
	x \in \R^d \setminus B_{r_1},
\end{equation}
we conclude that
		\begin{equation}\label{eq:estimJl1}
	\begin{aligned}
		J_1(t) & \le
		- \frac{k(2-k) \psi_1}{4}
		M^2
		+
		\frac{k(2-k) \psi_1}{4}
		\int_{ ( \R^d \times \R^d )
			\setminus	
			( B_{r_1} \times B_{r_1} )}
		u(t,x) u(t,y) \, \d x \, \d y \\
		& \le
		- \frac{k(2-k) \psi_1}{4}
		M^2
		+
		\frac{k(2-k) \psi_1}{2}
		\int_{\R^d}
		\int_{\R^d \setminus B_{r_1}}
		u(t,x) u(t,y) \, \d x \, \d y \\
		& =
		- \frac{k(2-k) \psi_1}{4}
		M^2
		+
		\frac{k(2-k) \psi_1}{2}
		M \frac{1}{r_1^{2-k}}
		\int_{\R^d \setminus B_{r_1}} r_1^{2-k} u(t,x) \, \d x \\
		& \le
		- \frac{k(2-k) \psi_1}{4}
		M^2
		+
		\frac{k(2-k) \psi_1}{2 r_1^{2-k}}
		M I(t).
	\end{aligned}
\end{equation}

In the next step, 
combining inequality~\eqref{eq:phi-upper} with relations \eqref{eq:R-incl}-\eqref{eq:phi-estim} we obtain
	\begin{equation}\label{eq:estim-difr}
	\begin{split}
		| J(t) - J_1(t) | & \le
		\frac{k(2-k)}{4}
		\int_{ ( \R^d \times \R^d )
			\setminus	
			( B_{r_1} \times B_{r_1} )}
		u(t,x) u(t,y) \left| \Psi(x,y) \right| \, \d x \, \d y
		\\
		& \le
		\frac{k(2-k) \psi_2}{4}
		\int_{ ( \R^d \times \R^d )
			\setminus	
			( B_{r_1} \times B_{r_1} )}
		u(t,x) u(t,y) \, \d x \, \d y \\
		& \le
		\frac{k(2-k) \psi_2}{2}
		\int_{\R^d}
		\int_{\R^d \setminus B_{r_1}}
		u(t,x) u(t,y) \, \d x \, \d y \\
		& \le
		\frac{k(2-k) \psi_2}{2 r_1^{2-k}}
		M I(t).
	\end{split}
\end{equation}
Gathering %identity~\eqref{eq:J} %, definition~\eqref{eq:def-D:n} 
 estimates~\eqref{eq:estimJl1}--\eqref{eq:estim-difr}, we obtain the differential inequality
	\begin{equation}\label{eq:diff-ineq}
		\frac{\d}{\d t} I(t)\le
		\varepsilon \D(u)(t)
		- \mu \omega M^2
		+ \omega M I(t)
		\quad \text{for all } t \ge 0,
	\end{equation}
	with the constants 
	\begin{equation}\label{eq:def-mu-omega}
    	\mu = \frac{\psi_1 r_1^{2-k}}{2(\psi_1 + \psi_2)} \quad\text{and}\quad 
		\omega =
		\frac{k(2-k)}{2r_1^{2-k}}
		\left( \psi_1 + \psi_2 \right),
	\end{equation}
	where we note that $\mu < 1$.
	
	Inequality~\eqref{eq:diff-ineq} can be rewritten as
	$$
	\frac{\d}{\d t}
	\left( I(t) e^{-\omega M t} \right)
	\le
	\left( \varepsilon \D(u)(t) - \mu \omega M^2 \right)
	e^{-\omega M t }
	$$
	and after integration with respect to time, we obtain
		\begin{equation*}
		I(T) e^{-\omega M T}
		- I(0) 
		\le
		\varepsilon
		\int_0^T
		\D(u)(t) e^{-\omega M t}
		\, \d t 
		-
		\mu M
		\left(1 - e^{-\omega M T} \right)
	\end{equation*}
	for each $T > 0$. Omitting the term $I(T) e^{-\omega M T}$, which is non-negative, we end up with the inequality
	\begin{equation}\label{eq:intD}
	\varepsilon
	\int_0^T
	\D(u)(t) e^{-\omega M t}
	\, \d t
	\ge
	\mu M
	\left(1 - e^{-\omega M T} \right)
	- I(0).
\end{equation}

Assumption~\eqref{eq:u0-concentration} tells us that $I(0) < \mu M$. Therefore, for each $T > T_0$, where 
\begin{equation}\label{eq:def-T0}
T_0 = \frac{1}{\omega M}
\log \left(\frac{\mu M}{\mu M - I(0)} \right) > 0,
\end{equation}
we can define the number
\begin{equation}\label{eq:def-L}
L = \mu M
\left(1 - e^{-\omega M T} \right)
- I(0)
> 0,
\end{equation}
which completes the proof of inequality~\eqref{eq:estim-D}.
\end{proof}

%%%%%%%%%%%%%%%%%%%%%%%%%%%%%%%%%%%%%%%%%%%%%%%%%%%%%%%%%%%%%

\section{Proof of the main results}\label{sec:proof}

In this last section, we prove the main results of this work stated in Section~\ref{sec:mr} by  using the concentration property of the quantity $\D(u_\varepsilon)$ established in Lemma~\ref{lem:conc} combined with the $L^p$-estimates from Lemmas \ref{lem:estim-global-p}--\ref{lem:estim-global-p-d1-k12}.
However, first, we estimate $\D(u_\varepsilon)$ from above.

\begin{lem}\label{lem:Destim1}
Let the assumptions of Theorem \ref{thm:main} hold true. Let 
$u_\varepsilon=u_\varepsilon(t,x)$ be a radial solution to problem \eqref{eq:main} with  $\varepsilon>0$.
Moreover, assume that
\begin{equation}\label{kd:ass}
k \in (0,2) \text{ and } d \ge 2, \quad \text{or} \quad k \in (0,1) \text{ and } d = 1.
\end{equation}
For every $T > 0$ and every $\nu>0$, there exist numbers $\varepsilon^* > 0$, $C_1=C_1(d,K)>0$,
$C_2 = C_2(M,T,\nu)>0$, and $\gamma>0$ 
such that for all $\varepsilon \in (0, \varepsilon^*)$, the quantity $\D(u_\varepsilon)$ defined in \eqref{eq:def-D:n} satisfies 
\begin{equation}\label{eq:eps-D-d2}
\varepsilon
\int_0^T
\D(u_\varepsilon)(t)
\, \d t
\le
C_1 M T \nu^{-1}
+
C_2 \left(
\int_0^T \int_{B_{(\nu \varepsilon)^{1/k}}}
u_\varepsilon(t,x)
\, \d x \, \d t
\right)^\gamma.
\end{equation}
\end{lem}

\begin{proof} In the following calculations, we write $u = u_\varepsilon(t,x)$ to simplify the notation.
Moreover, by a slight abuse of notation, we write $u(t,x) = u(t,r)$ with $r = |x|$ to emphasize the radial symmetry.

	Integrating by parts in the definition of $\D(u_\varepsilon)$, using the positivity and radial symmetry of $u$, and properties of the function $\varphi$,
	we obtain
	\begin{equation}
		\begin{split}
			\int_{\R^d} \varphi(|x|^{2-k}) \Delta u(t,x) \, \d x
			= & - (2-k) \int_{\R^d} \varphi'
			( |x|^{2-k} ) \frac{x}{|x|^k} \cdot \nabla u(t,x) \, \d x \\
			= & - \sigma_d (2-k) \int_0^{+\infty} \varphi'
			( r^{2-k} ) r^{d-k} u_r(t,r) \, \d r \\
			= & \int_0^{+\infty}
			\Big(
			\sigma_d (2-k)^2 \varphi''
			( r^{2-k} ) r^{d+1-2k}
			\\
			& 
			+ \sigma_d (2-k)(d-k) \varphi'
			( r^{2-k} ) r^{d-k-1} \Big)
			u(t,r) \, \d r \\	
			& \quad
			- \sigma_d (2-k)
			\Big[
			\varphi'
			( r^{2-k} ) r^{d-k} u(t,r) \Big]_{r=0}^{r=+\infty}.
		\end{split}
	\end{equation}
Since $\varphi'' \le 0$ by property~\eqref{eq:phi-prop3}, the integral involving $\varphi''$ on the right-hand side is non-positive. Moreover, the boundary term vanishes,
because $\varphi'$ vanishes at infinity.
Hence
\begin{equation}\label{eq:D-parts}
		\begin{split}
			\int_{\R^d} \varphi(|x|^{2-k}) \Delta u(t,x) \, \d x
    			\le & \sigma_d (2-k)(d-k) \int_0^{3/2} u(t,r) r^{d-k-1} \, \d r \\
			= & (2-k)(d-k) \int_{B_{3/2}} \frac{u(t,x)}{|x|^k}\, \d x.
		\end{split}
	\end{equation}
    
We estimate the integral on the right-hand side of inequality~\eqref{eq:D-parts}.
First,  for  any $T > 0$ and $\nu>0$, applying the H{\"o}lder inequality twice, we get
\begin{equation*}
	\begin{split}
		\int_0^T \int_{B_{(\nu \varepsilon)^{1/k}}} &
		\frac{u(t,x)}{|x|^k} 
		\, \d x \, \d t \\
		\le &
		\left(
		\int_0^T \int_{B_{(\nu \varepsilon)^{1/k}}}
		|x|^{-k \alpha_1}
		\, \d x \, \d t
		\right)^\frac{1}{\alpha_1}
		\left(
		\int_0^T \int_{B_{(\nu \varepsilon)^{1/k}}}
		u(t,x)^{\beta_1}
		\, \d x \, \d t
		\right)^\frac{1}{\beta_1}
		\\
		\le &
		\left(
		\frac{\sigma_d T}{d-k \alpha_1}
		(\nu \varepsilon)^{d/k - \alpha_1}
		\right)^\frac{1}{\alpha_1} 
		\left(
		\int_0^T \int_{B_{(\nu \varepsilon)^{1/k}}}
		u(t,x)^{
			\left(
			\beta_1 - 1/\beta_2
			\right)
			\alpha_2
		}
		\, \d x \, \d t
		\right)^\frac{1}{\beta_1 \alpha_2} \\
		& \times
		\left(
		\int_0^T \int_{B_{(\nu \varepsilon)^{1/k}}}
		u(t,x)
		\, \d x \, \d t
		\right)^\frac{1}{\beta_1 \beta_2}
		,
	\end{split}
\end{equation*}
for any exponents $\alpha_1, \beta_1, \alpha_2, \beta_2 \ge 1$ satisfying $k \alpha_1 < d$ and $1/\alpha_i + 1/\beta_i = 1$ for $i \in \{1,2\}$.

We set $p = (\beta_1 - 1/\beta_2) \alpha_2$ (observe that $p>1$)
and apply the $L^p$-estimates from Lemma~\ref{lem:estim-global-p} or Lemma~\ref{lem:estim-global-pk12}, depending on whether  $k \in (0,1]$ or $k\in (1,2)$.
%\retka{Since we do not have $p \ge 2$, we need either to change the lemmas or to interpolate simply by H{\"o}lder here (I prefer the second option).}
In both cases, we choose $\varepsilon > 0$ sufficiently small so that the
dominant contribution in the corresponding estimate is given by the term
$CM^{1+\eta/k}\varepsilon^{-\eta/k }$.
%\retka{Make it move to before \eqref{epsilon:ass1}?}
Thus,
\begin{equation*}
	\begin{split}
		\int_0^T \int_{B_{(\nu \varepsilon)^{1/k}}}
		\frac{u(t,x)}{|x|^k} 
		\, \d x \, \d t
		\le &
		\left(
		\frac{\sigma_d T}{d-k \alpha_1}
		(\nu \varepsilon)^{d/k - \alpha_1}
		\right)^\frac{1}{\alpha_1} 
		T^\frac{1}{\beta_1 \alpha_2}
		\left(
		\sup_{t \in [0,T]}
		\|u(t)\|_{p}
		\right)^{1 - \frac{1}{\beta_1 \beta_2}} \\
		& \times
		\left(
		\int_0^T \int_{B_{(\nu \varepsilon)^{1/k}}}
		u(t,x)
		\, \d x \, \d t
		\right)^\frac{1}{\beta_1 \beta_2} \\
		\le &
		C
		T^{\frac{1}{\alpha_1} + \frac{1}{\beta_1 \alpha_2}}
		\nu^{\frac{d}{k \alpha_1} - 1}
		\varepsilon^{\frac{d}{k \alpha_1} - 1 - \frac{\eta}{k} + \frac{\eta}{k \beta_1 \beta_2}} \\
		& \times
		\left(
		\int_0^T \int_{B_{(\nu \varepsilon)^{1/k}}}
		u(t,x)
		\, \d x \, \d t
		\right)^\frac{1}{\beta_1 \beta_2},
	\end{split}
\end{equation*}
where the number $C > 0$ is independent of $T$, $\nu$, and $\varepsilon$.

We recall that $\eta = d(p-1)/p$. Thus, the exponent of $\varepsilon$ equals
$$
-1+\frac{d}{k}
\left(
\frac{1}{\alpha_1}+\frac{p-1}{p}
    \left( \frac{1}{\beta_1 \beta_2}-1 \right) \right).
$$
Using the relations for $\alpha_i$ and  $\beta_i$, along with the definition of $p$, one can verify that this exponent is exactly $-1$. Therefore
\begin{equation}\label{eq:uxk1}
	\int_0^T \int_{B_{(\nu \varepsilon)^{1/k}}}
	\frac{u(t,x)}{|x|^k} 
	\, \d x \, \d t
	\le
	C
	T^\alpha
	\nu^\beta
	\varepsilon^{-1}
	\left(
	\int_0^T \int_{B_{(\nu \varepsilon)^{1/k}}}
	u(t,x)
	\, \d x \, \d t
	\right)^\gamma
\end{equation}
for some strictly positive numbers $\alpha, \beta$, and $\gamma$.

On the other hand, for arbitrary $\nu > 0$ and sufficiently small $\varepsilon > 0$ satisfying 
\begin{equation}\label{epsilon:ass1}
(\nu \varepsilon)^{1/k} < 3/2, 
\end{equation}
using the conservation of mass~\eqref{mass}, we estimate
\begin{equation}\label{eq:uxk2}
\begin{split}
	\int_0^T
	\int_{B_{3/2}
		\setminus
		B_{(\nu \varepsilon)^{1/k}}
	}
	\frac{u(t,x)}{|x|^k} 
	\, \d x \, \d t
    & \le
	\frac{1}{\nu \varepsilon}
	\int_0^T
	\int_{B_{3/2}
		\setminus
		B_{(\nu \varepsilon)^{1/k}}
	}
	u(t,x)
	\, \d x \, \d t \\
	& \le \frac{MT}{\nu \varepsilon}.
    \end{split}
\end{equation}
Combining estimates~\eqref{eq:D-parts}, \eqref{eq:uxk1} and \eqref{eq:uxk2}, we obtain inequality~\eqref{eq:eps-D-d2}
with 
$C_1=(2-k)(d-k)>0$.
\end{proof}

The remaining case, not covered by assumptions \eqref{kd:ass}, is considered in the following lemma.

\begin{lem}\label{lem:Destim2}
Let $k\in [1,2)$ and  $d=1$.
Let the assumptions of Theorem \ref{thm:main} hold true and let
$u_\varepsilon=u_\varepsilon(t,x)$ be a radial solution to problem \eqref{eq:main} with  $\varepsilon>0$.
Then, for every $p>1/(2-k)$, there exist numbers $C(M)>0$ and $\varepsilon^*>0$ such that for all $\varepsilon\in (0, \varepsilon^*)$ and all $T>0$ and all $\nu>0$,
\begin{equation}\label{D:est:d1:k12}
  \nu  \varepsilon \int_0^T \D (u_\varepsilon)(t)\,\d t \le  (2-k) \int_0^T\int_{-(\nu \varepsilon)^{1/k}}^{(\nu \varepsilon)^{1/k}} u_\varepsilon(t,x) \, \d x\, \d t
        + C(M) \nu^{\frac{1}{k}\left(2-\frac1p\right)}T.
\end{equation}

% For every $T > 0$ and every $\nu>0$, there exist numbers $\varepsilon^* > 0$, $C_1$, $C_2 = C_2(M,T,\nu)$ and $\gamma>0$ 
% such that for all $\varepsilon \in (0, \varepsilon^*)$

\end{lem}

\begin{proof}
Once again, we write $u = u_\varepsilon(t,x)$ to simplify the notation.
First, we integrate by parts in the definition of $\D(u)(t)$:
\begin{equation}\label{eq:parts1}
\int_{\R} \varphi(|x|^{2-k}) u_{xx}(t,x) \, \d x
			= - 2 (2-k) \int_{0}^{3/2} x^{1-k} \varphi'
			( x^{2-k} ) u_x(t,x) \, \d x.
\end{equation}
Now, for arbitrary parameters $0<\delta<\gamma <1/2$, integrating by parts once more and using the properties of the function $\varphi$ in~\eqref{eq:phi-prop}-\eqref{eq:phi-prop3}, we obtain 
\begin{equation}\label{eq:parts2}
    \begin{split}
    \int_{\delta}^{3/2} x^{1-k} \varphi'
			( x^{2-k} ) u_x(t,x) \, \d x = &
            -\delta^{1-k} u(t,\delta) 
            + (k-1) \int_{\delta}^{\gamma} x^{-k} u(t,x) \, \d x \\
            & + (k-1) \int_{\gamma}^{3/2} x^{-k} \varphi'
			( x^{2-k} ) u(t,x) \, \d x \\
            & - (2-k) \int_{\delta}^{3/2} x^{2-2k} \varphi''
			( x^{2-k} ) u(t,x) \, \d x.
    \end{split}
\end{equation}
Since
\begin{equation}%\label{eq:parts3}
    %\begin{split}
        \int_{\delta}^{\gamma} x^{-k} u(t,x) \, \d x
        %& = \int_{\delta}^{1/2} \frac{u(t,x) - u(t,\delta)}{x^k} \, \d x + \int_{\delta}^{1/2} x^{-k} u(t,\delta) \, \d x \\ &
        = \int_{\delta}^{\gamma} \frac{u(t,x) - u(t,\delta)}{x^k} \, \d x - \frac{u(t,\delta)}{k-1} \left(\gamma^{1-k} - \delta^{1-k} \right),
    %\end{split}
\end{equation}
substituting this identity into previous calculations~\eqref{eq:parts2} and passing to the limit $\delta \to 0$, we obtain
\begin{equation*}%\label{eq:parts4}
    \begin{split}
        \int_{\R} \varphi(|x|^{2-k}) u_{xx}(t,x) \, \d x
        			= & 2 (2-k) \left( u(t,0) \gamma^{1-k} - (k-1) \int_{0}^{\gamma} \frac{u(t,x) - u(t,0)}{x^k} \, \d x \right. \\
                    & - (k-1) \int_{\gamma}^{3/2} x^{-k} \varphi'
			( x^{2-k} ) u(t,x) \, \d x \\
            & \left. + (2-k) \int_{0}^{3/2} x^{2-2k} \varphi''
			( x^{2-k} ) u(t,x) \, \d x \right).
    \end{split}
\end{equation*}
The last two terms on the right-hand side are non-positive because $k \in (1,2)$ and because of the properties~\eqref{eq:phi-prop}-\eqref{eq:phi-prop3} of the function $\varphi$. Therefore
\begin{equation*}
\nu \varepsilon\int_{\R} \varphi(|x|^{2-k}) u_{xx}(t,x) \, \d x
        			\le 2 (2-k) \left( u(t,0) \nu \varepsilon \gamma^{1-k} - (k-1) \nu \varepsilon \int_{0}^{\gamma} \frac{u(t,x) - u(t,0)}{x^k} \, \d x \right).
\end{equation*}

Now, we choose $\gamma$ such that
\begin{equation}\label{gamma}
\nu \varepsilon \gamma^{1-k}= (\nu \varepsilon)^{1/k} \quad\text{which implies} \quad\gamma=(\nu\varepsilon)^{1/k}.
\end{equation}
Observe that
$$
\int_{-(\nu \varepsilon)^{1/k}}^{(\nu \varepsilon)^{1/k}} u(t,x) \, \d x
= 2 (\nu \varepsilon)^{1/k} u(t,0)
+ \int_{-(\nu \varepsilon)^{1/k}}^{(\nu \varepsilon)^{1/k}} \int_0^x u_y(t,y) \, \d y \, \d x,
$$
where, by H{\"o}lder's inequality and Lemma \ref{lem:estim-global-p-d1-k12} (with $p\geq 2$ and with sufficiently small $\varepsilon>0$):
%\retka{We also need to impose $p \ge 2$}
\begin{equation}
    \begin{split}
    \int_{-(\nu \varepsilon)^{1/k}}^{(\nu \varepsilon)^{1/k}} \int_0^x u_y(t,y) \, \d y \, \d x
    &\le \int_{-(\nu \varepsilon)^{1/k}}^{(\nu \varepsilon)^{1/k}} |x|^{1-\frac1p} \|u_x(t)\|_p \, \d x \\
    &\le  C(M) \varepsilon^{-\frac1k\left(2-\frac1p\right)} (\nu \varepsilon)^{\frac{1}{k}\left(2-\frac1p\right)} 
    =C(M) \nu^{\frac{1}{k}\left(2-\frac1p\right)}.
    \end{split}
\end{equation}

Next, we estimate
\begin{equation}
  \begin{split}  
|u(t,x) - u(t,0)| &=
    \left|
    \int_0^1 \frac{\d}{\d s} u(t,sx) \, \d s
    \right|
    =
    \left|
    \int_0^1 x  u_x(t,sx)  \, \d s
    \right| \\
&\le |x| \left( \int_0^1 |u_x(t,sx)|^p \, \d s \right)^{1/p} \le  |x|^{1-1/p} \| u_x(t) \|_p.
\end{split}
\end{equation}
Thus, for every $p>1/(2-k)$, by Lemma \ref{lem:estim-global-p-d1-k12} and by the definition of $\gamma$ in \eqref{gamma}:
\begin{equation}
    \begin{split}
\nu\varepsilon\int_0^{\gamma}
\frac{|u(t,x) - u(t,0)|}{x^k} \, \d x 
&\le \nu \varepsilon \| u_x(t) \|_p \int_0^{\gamma} x^{1-\frac1p-k} \, \d x \\
&\le C \nu\varepsilon\varepsilon^{-\frac1k\left(2-\frac1p\right)} \gamma^{2-\frac1p-k}= C\nu^{\frac1k\left(2-\frac1p\right)} .
    \end{split}
\end{equation}
Combining the above estimates and choosing $\gamma$ as in \eqref{gamma}, we have 
\begin{equation}
    \begin{split}
        \nu\varepsilon\D(u)(t) &\le 
        2 (2-k) \left( u(t,0) \nu\varepsilon\gamma^{1-k} - (k-1) \nu\varepsilon\int_{0}^{\gamma} \frac{u(t,x) - u(t,0)}{x^k} \, \d x \right)\\
        & \le (2-k) \int_{-(\nu \varepsilon)^{1/k}}^{(\nu \varepsilon)^{1/k}} u(t,x) \, \d x
        + C(M) \nu^{\frac{1}{k}\left(2-\frac1p\right)}.
    \end{split}
\end{equation} 
The proof is completed by integrating both sides of this inequality with respect to $t\in [0,T]$. 
\end{proof}

%%%%%%%%%%%%%%%%%%%%%%%%%%    

We are in a position to prove the main results of this work.

\begin{proof}[Proof of Theorem \ref{thm:main}.]
\textit{The case 
$k \in (0,2) \text{ and } d \ge 2$, 
or
$k \in (0,1) \text{ and } d = 1$.
}
Combining 
 estimate \eqref{eq:eps-D-d2} from Lemma~\ref{lem:Destim1}
with 
 inequality  \eqref{eq:estim-D}
 from the concentration Lemma~\ref{lem:conc} 
we obtain
\begin{equation}\label{eq:uxk3}
L-C_1MT\nu^{-1} 
	\le
	C_2
	\left(
	\int_0^{T} \int_{B_{(\nu \varepsilon)^{1/k}}}
	u(t,x)
	\, \d x \, \d t
	\right)^\gamma.
\end{equation}
We choose
$
\nu > {M T C_1}/L,
$
which ensures that the left-hand side  of inequality~\eqref{eq:uxk3} is positive. 
We have completed  the proof of inequality~\eqref{main_estimate}
for such $\nu$
and for all sufficiently small $\varepsilon>0$, namely, those satisfying inequality~\eqref{epsilon:ass1} and such that 
the term $CM^{1+\eta/k}\varepsilon^{-\eta/k }$ dominates the remaining terms in the estimates of the $L^p$-norms of solutions.

{\it The case $k\in [1,2)$ for $d=1$.} 
Here, the proof is analogous: it suffices to apply estimate~\eqref{D:est:d1:k12} to
inequality~\eqref{eq:estim-D}
 from the concentration Lemma~\ref{lem:conc}.
Note that here we need $\nu$ to be small since $(2-1/p)/k>1$.
\end{proof} 

\begin{proof}[Proof of Corollary~\ref{cor:lp-lower}]
	For $p\in \left[1,\infty \right)$, by Theorem~\ref{thm:main} and the H{\"o}lder inequality, we obtain
	\begin{align*}
		C &  \le \int_0^{T} \int_{B_{(\nu\varepsilon)^{1/k}}} u(t,x) \, \d x\, \d t \\ 
		&  \le \left((\nu\varepsilon)^{d/k} \frac{\pi^{d/2}}{\Gamma(d/2+1)} \right)^{\frac{p-1}{p}}
        \int_0^{T} \left( \int_{B_{(\nu\varepsilon)^{1/k}}} u(t,x)^p\, \d x \right)^{1/p} \, \d t.
	\end{align*} 
 The estimates for $p=\infty$ are similar.
\end{proof}

 %    solution $u_{a,b}$ with $t = 0$ and we use the dependence of the initial datum on numbers $a,b \ge 1$ to obtain
	% $$
	% I_{a,b}'(0) \le a b^{k-1} \D(u)(0)
	% - a^2 b^{-2} \mu \omega  M^2
	% + a^2 b^{k-4} \omega  M I(0),
	% $$
	% where we use quantities $\D$, $M$ and $I$ after rescaling.
	
	% Imposing that the right hand side of this inequality is negative, we obtain the following condition,
	% $$
	% b^{k+1} |\D(u)(0)| + a \omega M \left( b^{k-2} I(0) - \mu M \right) < 0,
	% $$
	% which is satisfied for sufficiently large $b \ge 1$ and $a \ge 1$.
	% Consequently, by the continuity of the functions $u_{a,b}$ and $I_{a,b}$ in time, condition $I_{a,b}'(t) < 0$ is satisfied for every $t \in [0,T]$ with some $T = T(a,b) > 0$.

%%%%%%%%%%%%%%%%%%%%%%%%%%%%%%%%%%%%%%%%%%%%%%%%%%%%%%%%%%%%%

\appendix

\section{Estimates of the integration kernel}\label{appendix}

\subsection{Homogeneous kernel}

We first describe the properties of the homogeneous integration kernel
\begin{equation}\label{eq:def-Phi0}
	\Phi_0(x,y) \equiv
	\left(
	\frac{x}{|x|^k} -
	\frac{y}{|y|^k}
	\right)
    \cdot
	\frac{x - y}{|x - y|^{2 - k}},
\end{equation}
with the parameter \( k \in (0,2) \), defined for all $ (x,y) \in \R^d \times \R^d \setminus Y $, where
$$
	Y \equiv \left\{ (x,y) \in \R^d \times \R^d : x = 0 \ \vee\ y = 0 \ \vee\ x = y \right\}.
$$
We also consider its "symmetrized" version
\begin{equation}\label{eq:def-Psi0}
	\Psi_{0}(x,y) \equiv \Phi_{0}(x,y) + \Phi_{0}(x,-y),
\end{equation}
defined analogously for all \( (x,y) \in \R^d \times \R^d \setminus Z \)
% with $d \ge 1$
, where
$$
	Z \equiv \left\{ (x,y) \in \R^d \times \R^d : x = 0 \ \vee\ y = 0 \ \vee\ x = y \ \vee\ x = -y \right\}.
$$

Before proving all the properties, we recall without proof the well-known inequalities used throughout this appendix.

\begin{lem}\label{lem:est-dk-upp}
	Let $d \ge 1$, $a,b \in \R^d$ and $\alpha \ge 0$.
    The following inequality holds:
	$$ %\begin{equation}\label{eq:est-upp}
    |a \pm b|^\alpha
    \le C_\alpha \left( |a|^\alpha + |b|^\alpha \right),
    $$ %\end{equation}
	where $C_\alpha = 1$ if $\alpha \in [0,1]$
    and $C_\alpha = 2^{\alpha-1}$ if $\alpha > 1$.
\end{lem}

\begin{lem}\label{lem:est-1k-low}
	Let $a,b \in \R$ and $\alpha \ge 0$.
    For $ab \ge 0$, the following inequality holds:
	$$
    |a + b|^\alpha \ge c_\alpha \left( |a|^\alpha + |b|^\alpha \right),
	$$
	where $c_\alpha = 2^{\alpha-1}$ if $\alpha \in [0,1)$ and $c_\alpha = 1$ if $\alpha \ge 1$.
\end{lem}

In the case $\alpha \in (0,1)$, one can consider a function $t \mapsto t^\alpha$. Then Lemma~\ref{lem:est-dk-upp} is the subadditivity property of this function, from which one can conclude that
\begin{equation}\label{eq:subadd}
    a^\alpha - b^\alpha \le (a-b)^\alpha
\end{equation}
for $a \ge b \ge 0$.
Also, Lemma~\ref{lem:est-1k-low} can be seen as a generalization of Jensen's inequality.

Now, we begin with some preliminary estimates.

\begin{lem} \label{lem:Holder-2k}
Let $d \ge 1$ and $\beta \in [0,1]$. Then
\begin{equation}\label{eq:Holder-k}
\left|\frac{x}{|x|^\beta} - \frac{y}{|y|^\beta}\right| \le 2^\beta |x-y|^{1-\beta}
\end{equation}
for all $x, y \in \R^d$ such that $x \neq 0$ and $y \neq 0$.
\end{lem}

\begin{proof}
    The cases $\beta = 0$ and $\beta = 1$ are obvious. From now on, we assume that $\beta \in (0,1)$.

    Let $d = 1$. Note that the left-hand side of inequality~\eqref{eq:Holder-k} is symmetric with respect to the lines $y = x$ and $y = -x$. First, we consider the case $xy > 0$. From the symmetry, we assume without loss of generality that $x \ge y > 0$. Then the assertion follows from inequality~\eqref{eq:subadd} with $a = x$, $b = y$, $\alpha = 1 - \beta$ and the fact that $2^\beta > 1$.
    In the case $xy < 0$, once again by the symmetry of the left-hand side, we assume $x > 0 > y$. Denoting $z := -y$, inequality~\eqref{eq:Holder-k} becomes Lemma~\ref{lem:est-1k-low} with $a = x$, $b = z$ and $\alpha = 1 - \beta$.    
    
    Now, let $d \ge 2$. We consider the squared form of the left-hand side of inequality~\eqref{eq:Holder-k},
    $$
    \left|\frac{x}{|x|^\beta} - \frac{y}{|y|^\beta}\right|^2
    = |x|^{2-2\beta} + |y|^{2-2\beta} - 2 |x|^{-\beta} |y|^{-\beta} x \cdot y.
    $$
    In the first case, we take $x \cdot y < 0$. Using Lemma~\ref{lem:est-1k-low} with $a = |x|^2$, $b = |y|^2$ and $\alpha = 1 - \beta$, we have
    $$
    \left|\frac{x}{|x|^\beta} - \frac{y}{|y|^\beta}\right|^2
    \le
    2^\beta \left( |x|^2 + |y|^2 \right)^{1-\beta} - 2 |x|^{-\beta} |y|^{-\beta} x \cdot y.
    $$
    We denote $z := -y$. Then $x \cdot z > 0$ and using the Cauchy-Schwarz inequality,
    %$x \cdot z \le |x||z|$
    we get
    $$
    \left|\frac{x}{|x|^\beta} + \frac{z}{|z|^\beta}\right|^2
    \le
    2^\beta \left( |x|^2 + |z|^2 \right)^{1-\beta}
    + 2 (x \cdot z)^{1-\beta}.
    $$
    By factoring out the number $2^\beta$,
    we apply Lemma~\ref{lem:est-1k-low} with
    $a = |x|^2 + |z|^2$, $b = 2 x \cdot z$ and $\alpha = 1 - \beta$ to obtain
    $$
    \left|\frac{x}{|x|^\beta} + \frac{z}{|z|^\beta}\right|^2
    \le
    2^{2 \beta} \left( |x|^2 + |z|^2 + 2 x \cdot z \right)^{1-\beta}
    = 2^{2 \beta} \left| x+z \right|^{2-2\beta},
    $$
    which is inequality~\eqref{eq:Holder-k}.

We are left with the case $x \cdot y \ge 0$. We write $x \cdot y = s |x||y|$ with $s = \cos \theta \in [0,1]$, where~$\theta$ is the angle between $x$ and $y$. 
Observe the identity
\begin{equation*}
    |x-y|^2 = s ( |x| - |y| )^2 + (1-s)(|x|^2 + |y|^2),
\end{equation*}
which expresses $|x-y|^2$ as a convex combination of $( |x| - |y| )^2$ and $|x|^2 + |y|^2$. 
Applying Jensen's inequality to the concave function $t \mapsto t^{1-\beta}$ yields
\begin{equation*}
    |x-y|^{2-2\beta} =
    \left( |x-y|^2 \right)^{1-\beta} \ge
    s \left| |x| - |y| \right|^{2-2\beta}
    + (1-s) \left( |x|^2 + |y|^2 \right)^{1-\beta}.
\end{equation*}
First, note that
$$
\left| |x| - |y| \right|^{2-2\beta} = 
\left( |x|^2 + |y|^2 - 2|x||y| \right)^{1-\beta}.
$$
Therefore, applying inequality~\eqref{eq:subadd} with $a = |x|^2 + |y|^2$, $b = 2|x||y|$, and $\alpha = 1-\beta$ yields
$$
\left| |x| - |y| \right|^{2-2\beta} \ge 
\left( |x|^2 + |y|^2 \right)^{1-\beta} - 2^{1-\beta} |x|^{1-\beta} |y|^{1-\beta}.
$$
Moreover, we estimate $\left( |x|^2 + |y|^2 \right)^{1-\beta}$ from below using Lemma~\ref{lem:est-1k-low} with $a = |x|^2$, $b = |y|^2$, and $\alpha = 1 - \beta$. Combining all these estimates, we obtain
\begin{equation}
    \begin{split}
        |x-y|^{2-2\beta} & \ge (|x|^2 + |y|^2)^{1-\beta} - 2^{1-\beta} |x|^{1-\beta} |y|^{1-\beta} s\\
        & \ge 2^{-\beta} |x|^{2- 2\beta} + 2^{-\beta} |y|^{2- 2\beta} - 2^{1-\beta} |x|^{1-\beta} |y|^{1-\beta} s.
    \end{split}
\end{equation}
Multiplying the last inequality by $2^\beta$, we obtain
$$
|x|^{2-2\beta} + |y|^{2-2\beta} - 2 |x|^{1-\beta} |y|^{1-\beta} s \le 2^\beta |x-y|^{2-2\beta},
    $$
which implies inequality~\eqref{eq:Holder-k}, since $2^{\beta} \le 2^{2 \beta}$. 
\end{proof}

\begin{lem}\label{lem:0A}
Let $d \ge 1$ and $k \in (0,1]$.
Then the function $\Phi_0$, defined by formula \eqref{eq:def-Phi0}, satisfies estimate
\begin{equation}\label{eq:Fi0-estim}
0 \le
\Phi_0(x,y)\le 2^k
\end{equation}
for all $ (x,y) \in \R^d \times \R^d \setminus Y $.
\end{lem}

\begin{proof}
By a direct calculation, using the inequality $x \cdot y \le |x||y|$,  we obtain
\begin{equation}\label{ineq:A:lemma}
\begin{split}
\left(\frac{x}{|x|^{k}}-\frac{y}{|y|^{k}}\right)\cdot(x-y)
& = |x|^{2-k}-\frac{x\cdot y}{|x|^{k}}-\frac{x\cdot y}{|y|^{k}}+|y|^{2-k} \\
& \ge |x|^{2-k}-|x|^{1-k}|y|-|x||y|^{1-k}+|y|^{2-k} \\
& = (|x|^{1-k}-|y|^{1-k})(|x|-|y|).
\end{split}
\end{equation}
Since $k \in (0,1]$, the function $t \mapsto t^{1-k}$ is non-decreasing for $t \ge 0$. Therefore, the terms $|x|^{1-k}-|y|^{1-k}$ and $|x|-|y|$ always have the same sign.
Since the denominator $|x-y|^{2-k}$ is strictly positive for $(x,y) \notin Y$, the lower bound $\Phi_0(x,y) \ge 0$ follows.

For $k=1$,
the function $\Phi_0$ is obviously bounded from above by 2.
To find the upper bound   for $k\in (0,1)$, it suffices to apply Lemma \ref{lem:Holder-2k} with $\beta = k$.
\end{proof}

\begin{rem}
When $k \in (1,2)$, the function $\Phi_0$ no longer satisfies estimate~\eqref{eq:Fi0-estim}.
We~therefore consider its "symmetrized" version $\Psi_0$, defined by formula~\eqref{eq:def-Psi0}, and in the following lemma we establish its lower bound for all $k \in (0,2)$.
% We note that the function $\Psi_0$ is also bounded from above on $\R^d\times \R^d \setminus Z$  by
% $\max\{2,2^k\}$ if $d = 1$
% and
% $2^{k/2 + 1}$ if $d \ge 2$.
% %, which can be proved similarly to {\red Lemma~\ref{lem:1A}, below}.
% However, in this work, we only require its lower bound.
\end{rem}

\begin{lem}\label{lem:1A}
Let $d \ge 1$ and $k \in (0,2)$. Then the function $\Psi_0$, defined by formula~\eqref{eq:def-Psi0}, satisfies the estimate
$$
\Psi_{0}(x,y) \ge c_k\quad 
$$
for all
$
(x,y) \in \mathbb{R}^d \times \mathbb{R}^d \setminus Z
$,
where $c_k = \min \{ 2^k, 2 \}$.
\end{lem}
\begin{proof}
We estimate from below the expression
%By direct calculation, we obtain:
%
\begin{equation}\label{Psi0:proof}
%\begin{split}
\Psi_{0}(x,y)
%&= \Phi_{0}(x,y)+\Phi_{0}(x,-y) \\ &
= \left(\frac{x}{|x|^{k}}-\frac{y}{|y|^{k}}\right)\cdot\frac{x-y}{|x-y|^{2-k}} + \left(\frac{x}{|x|^{k}}+\frac{y}{|y|^{k}}\right)\cdot\frac{x+y}{|x+y|^{2-k}}.
%\end{split}
\end{equation}

First, we consider $k\in (0,1]$. 
Following the argument used in estimate~\eqref{ineq:A:lemma} with the inequality $\pm x \cdot y \ge -|x||y|$, we have
$$
\left(\frac{x}{|x|^{k}} \pm \frac{y}{|y|^{k}}\right)\cdot(x \pm y) \ge 0.
$$
Due to the non-negativity of the numerators, one can estimate the denominators using Lemma~\ref{lem:est-dk-upp} with $a = x$, $b = y$ and $\alpha = 2 - k$, to obtain
\begin{equation}\label{eq:Psi0-lower-k01}
\Psi_{0}(x,y) \ge
\frac{2 \left(|x|^{2-k}+|y|^{2-k} \right)}
{2^{1-k} \left(|x|^{2-k}+|y|^{2-k} \right)} = 2^k
\end{equation}
for all $(x,y)\in\R^d\times\R^d \setminus Z$.

%%%%%%%
Next, we deal with the case $k\in (1,2)$. Substituting
\begin{equation}\label{eq:subst-zw}
    z = x + y, \ w = x - y
    \quad
    \iff
    \quad
    x = \frac{z+w}{2}, \ y = \frac{z-w}{2}
\end{equation}
into the definition~\eqref{Psi0:proof} of the function $\Psi_0$, we obtain
$$
2^{k-1}
\left[
    \left(
        \frac{z+w}{|z+w|^{k}}-\frac{z-w}{|z-w|^{k}}
    \right)
    \cdot
    \frac{w}{|w|^{2-k}}
    +
    \left(
        \frac{z+w}{|z+w|^{k}}+\frac{z-w}{|z-w|^{k}}
    \right)
    \cdot
    \frac{z}{|z|^{2-k}}
\right],
$$
which, after rearrangement, can be written as
$$
2^{k-1}
\left[
    \left(
        \frac{z}{|z|^{2-k}}-\frac{w}{|w|^{2-k}}
    \right)
    \cdot
    \frac{z-w}{|z-w|^{k}}
    +
    \left(
        \frac{z}{|z|^{2-k}}+\frac{w}{|w|^{2-k}}
    \right)
    \cdot
    \frac{z+w}{|z+w|^{k}}
\right].
$$
Note that the expression in square brackets corresponds to the definition of the function $\Psi_{0}$ with $k$ replaced by $2-k$. Since $2-k \in (0,1)$, we may use the previously proved lower bound~\eqref{eq:Psi0-lower-k01} to obtain
$$
\Psi_{0}(x, y) \ge 2^{k-1} 2^{2-k} = 2
$$
for all $(x,y)\in\R^d\times\R^d \setminus Z$.
\end{proof}

%%%%%%%%%%%%%%%%

\subsection{General kernel}

We recall the definition~\eqref{eq:def-Phi} of the integration kernel $\Phi$ that appears in the proof of Lemma~\ref{lem:conc}:
\begin{equation}\label{Phi:A}
\Phi(x,y) =
    \left(
        \varphi' ( |x|^{2-k} ) \frac{x}{|x|^k} -
        \varphi' ( |y|^{2-k} ) \frac{y}{|y|^k}
    \right)
    \cdot
    \frac{x-y}{|x-y|^{2-k}}h'(|x-y|^k),
\end{equation}
where $ (x,y) \in \R^d \times \R^d \setminus Y $ and
$$
	Y \equiv \left\{ (x,y) \in \R^d \times \R^d : x = 0 \ \vee\ y = 0 \ \vee\ x = y \right\}.
$$
Here, $k\in (0,2)$, the truncation function $\varphi$ is defined in Remark~\ref{rem:def-varphi} and $h=h(r)$ follows from the kernel definition in Assumption~\ref{ass:h}. We then introduce its "symmetrized" version,
\begin{equation}\label{Psi:A}
\Psi(x,y) \equiv \Phi(x,y) + \Phi(x,-y)
\end{equation}
for $(x,y)\in\R^d\times\R^d\setminus Z$, where
$$
	Z \equiv \left\{ (x,y) \in \R^d \times \R^d : x = 0 \ \vee\ y = 0 \ \vee\ x = y \ \vee\ x = -y \right\}.
$$

We begin the analysis of functions $\Phi$ and $\Psi$ with some comments on the function $h=h(r)$ that appears in Assumption~\ref{ass:h}.

\begin{rem}\label{rem:h}
    Note that under assumption~\eqref{ass:K4},
    $$
    \nabla K(x) = \frac{kx}{|x|^{2-k}}h'(|x|^k)
    $$
    and
    $$
    \Delta K(x) = k^2 |x|^{2k-2} h''(|x|^k) + \frac{k(k+d-2)}{|x|^{2-k}} h'(|x|^k).
    $$
    Combining the first relation with assumption~\eqref{ass:K1}, we obtain
    \begin{equation}\label{ass:K7}
        h'\in L^\infty (0,\infty).
    \end{equation}
    Moreover,  by the second relation and by assumption \eqref{ass:K3} for $k\in (1,2)$, we have
    \begin{equation}\label{ass:K10}
        |\cdot|^k h''\in L^\infty (0,\infty).
    \end{equation}
    We also notice that, for $k \in (1,2)$,
    assumption~\eqref{ass:K6} implies that $h'$ is uniformly continuous near the origin, and thus the following limit exists and is finite:
    \begin{equation}\label{ass:K8}
        \lim_{r\to 0} h'(r)=c_0,
    \end{equation}
    where $c_0\ge c_1>0$ by assumption~\eqref{ass:K5}. 
\end{rem}

In the following two lemmas, we show that the function $\Psi$ defined in~\eqref{Psi:A} is bounded both from below and above.

\begin{lem}\label{lem:Psi:lowerbound}
Let $d \ge 1$ and $k \in (0,2)$.
Consider the function $\Psi$ defined in~\eqref{Psi:A} for all $(x,y)\in\R^d\times\R^d\setminus Z$,
where the functions $h$ and $\varphi$ are given in Assumption~\ref{ass:h} and Remark~\ref{rem:def-varphi}, respectively. 
Then there exist constants $r_1 \in (0,1/2]$ and $\psi_1 > 0$ such that
\begin{equation}\label{Psi:lowerbound}
\Psi(x, y) \ge \psi_1% \quad \text{for all} \quad |x|, |y| \le r_1.
\end{equation}
for all $|x|, |y| \le r_1$.
\end{lem}
\begin{proof}
Recall from definition~\eqref{eq:phi-prop} that $\varphi'(|x|^{2-k})=1$ for $|x|\le 1/2$.
Thus, for all $|x|,|y| \le 1/2$,
the function $\Psi$ can be written as
\begin{equation}
\begin{split}
\Psi(x,y) = &\left(\frac{x}{|x|^k} - \frac{y}{|y|^k}\right) \cdot \frac{x-y}{|x-y|^{2-k}} h'(|x-y|^k) \\
& \quad + \left(\frac{x}{|x|^k} + \frac{y}{|y|^k}\right) \cdot \frac{x+y}{|x+y|^{2-k}} h'(|x+y|^k).
\end{split}
\end{equation}
For $k\in (0,1]$, using the argument from the proof of Lemma \ref{lem:0A}, we have
$$
\left(\frac{x}{|x|^k} \pm \frac{y}{|y|^k}\right)
\cdot \frac{x\pm y}{|x \pm y|^{2-k}} \ge 0
% \quad\text{and}\quad
% \left(\frac{x}{|x|^k} + \frac{y}{|y|^k}\right) \cdot \frac{x+y}{|x+y|^{2-k}} \ge 0
.
$$
Let $r_1 \in (0,1/2]$ be a number such that $(2r_1)^k\le r_0$, where the constant $r_0 > 0$ is the one in assumption~\eqref{ass:K5}.
Hence, for $|x|, |y|\le r_1$, $|x\pm y|^k\le r_0$ and $h'(|x\pm y|^k)\ge c_1$.
Therefore, applying Lemma \ref{lem:1A} we immediately obtain that since $\Psi_0(x,y) \ge c_1$, we have 
$\Psi(x,y) \ge c_1 c_k \equiv \psi_1$.
\\
Now, let $k \in (1, 2)$. To find a lower bound for $\Psi$, we first estimate the difference $\Psi - c_0 \Psi_0$, where $c_0 > 0$ is the constant from relation~\eqref{ass:K8}.
For $|x|,|y|\le 1/2$, we have
\begin{equation}\label{PsiPsi0}
\begin{split}
\Psi(x,y)& - c_0 \Psi_0(x,y)  \\
= & \left(\frac{x}{|x|^k} - \frac{y}{|y|^k}\right) \cdot \frac{x-y}{|x-y|^{2-k}}
\left(
    h' ( |x-y|^k ) - c_0
\right) \\
& \quad + \left(\frac{x}{|x|^k} + \frac{y}{|y|^k}\right) \cdot \frac{x+y}{|x+y|^{2-k}}
\left(
    h' ( |x+y|^k ) - c_0
\right) \\
= & 
    \frac{x}{|x|^k} \cdot
        \left( 
        \frac{y+x}{|y+x|^{2-k}}
            \left( 
            h' (|y+x|^k) - c_0
            \right) 
        - \frac{y-x}{|y-x|^{2-k}}
            \left( 
            h' (|y-x|^k) - c_0
            \right)
        \right)
    \\
& \quad +
    \frac{y}{|y|^k} \cdot
        \left(
        \frac{x+y}{|x+y|^{2-k}}
            \left( 
            h' (|x+y|^k) - c_0
            \right) 
        - \frac{x-y}{|x-y|^{2-k}}
            \left( 
            h' (|x-y|^k) - c_0
            \right)
        \right)
.
\end{split}
\end{equation}

We consider the first term on the right-hand side of equality~\eqref{PsiPsi0}.
We add and subtract the term
$
(y-x)
(h' (|y+x|^k) - c_0)
/|y-x|^{2-k}
$
and estimate by
\begin{equation}\label{first-term:A}
    \begin{split}
    & |x|^{1-k} \left(
        \left|
            %\left(
            \frac{y+x}{|y+x|^{2-k}} - \frac{y-x}{|y-x|^{2-k}}
        \right| %\right)
        \left| %\left( 
            h' (|y+x|^k) - c_0
            %\right)
        \right| \right. \\
    & \hspace{6cm} + 
        % \left| \frac{y-x}{|y-x|^{2-k}} \right|
        |y-x|^{k-1}
        \left|
        h' (|y+x|^k) - h' (|y-x|^k)
        \right|
        \bigg).
    \end{split}
\end{equation}
By Lemma~\ref{lem:Holder-2k}
with $\beta=2-k\in (0,1)$,
%for $k \in (1, 2)$
we have
\begin{equation}
    \begin{split}
        \left|
            %\left(
            \frac{y+x}{|y+x|^{2-k}} - \frac{y-x}{|y-x|^{2-k}}
        \right| %\right)
        \le 2^{2-k} |2x|^{k-1}= 2|x|^{k-1}.
    \end{split}
\end{equation}
Moreover, by Lemma~\ref{lem:est-dk-upp} with $\alpha = k-1 \in (0,1)$,
$$
|y-x|^{k-1} \le |x|^{k-1}+|y|^{k-1}.
$$

The function $h''$ is bounded on $(0,r_0]$ by assumption~\eqref{ass:K6}. Thus for $r_1 \in (0,1/2]$ such that $(2r_1)^k\le r_0$, and for $|x|, |y| \le r_1$, we obtain
$$
\left| h'(|y+x|^k) - c_0 \right| =
\left| h'(|y+x|^k) - h'(0) \right| \le c_2|y+x|^k \le c_2 2^{k-1} \left( |x|^k+|y|^k \right),
$$
where we used
Lemma~\ref{lem:est-dk-upp} with $\alpha = k \in (1,2)$.
Analogously, 
$$
\left| h'(|y+x|^k) - h'(|y-x|^k) \right| 
    \le c_2 \left| |y+x|^k - |y-x|^k \right|
$$
and by the Mean Value Theorem for the function $t \mapsto |t|^k$, along with
Lemma~\ref{lem:est-dk-upp} with $\alpha = k-1 \in (0,1)$, we have
$$
\left| |y+x|^k - |y-x|^k \right|
\le 2 k|x| \left( |x|^{k-1}+|y|^{k-1} \right).
$$
% \begin{equation}
%     \begin{split}
%     \left| h'(|y+x|^k) - h'(|y-x|^k) \right| 
%     & \le c_2 \left| |y+x|^k - |y-x|^k \right| \\
%     & \le 2 c_2 k|x| \left( |x|+|y| \right)^{k-1}\\
%     & \le 2 c_2 k|x| \left( |x|^{k-1}+|y|^{k-1} \right). 
%     \end{split}
% \end{equation}

Combining all these estimates for $|x|, |y|\le r_1$, the expression in~\eqref{first-term:A} is bounded above by
\begin{equation}
    |x|^{1-k} \left(
        2^k c_2
        |x|^{k-1}
        \left( |x|^k+|y|^k \right)
        +
        2 c_2 k|x|
        \left( |x|^{k-1}+|y|^{k-1} \right)^2
        \right)
        \le
        % \left( 2^{k+1} + 8k \right) c_2 
        C r_1^k,
\end{equation}
with some $C > 0$.
The second term on the right-hand side of equality~\eqref{PsiPsi0} is estimated analogously, leading to the final inequality
$$
\left| \Psi(x,y) - c_0 \Psi_0(x,y) \right|
\le
% \left( 2^{k+2} + 16k \right) c_2
C r_1^k
% \quad\text{for} \quad
% |x|,|y| \le r_1,
$$
for $|x|,|y| \le r_1$,
where $r_1\in (0,1/2]$ is an arbitrary number satisfying $(2r_1)^k\le r_0$.

Next, we choose $r_1 > 0$ sufficiently small such that
$
%\left( 2^{k+2} + 16k \right) c_2 
C r_1^k \le c_0 c_k/2,$
where the constant $c_k>0$ comes from Lemma~\ref{lem:1A}. This choice completes the proof of inequality~\eqref{Psi:lowerbound} in the following way:
\begin{equation}
\Psi(x,y) \ge c_0 \Psi_0(x,y) -|\Psi(x,y)-c_0\Psi_0(x,y)|\ge \frac{c_0 c_k}{2} \equiv \psi_1
\end{equation}
for all $|x|,|y|\le r_1$.
\end{proof}

\begin{lem}\label{lem:Psi:upperbound}
Let $d \ge 1$ and $k\in (0,2)$. 
Consider the function $\Psi$ defined in~\eqref{Psi:A} for all $(x,y)\in\R^d\times\R^d\setminus Z$,
where the functions $h$ and $\varphi$ are given in Assumption~\ref{ass:h} and Remark~\ref{rem:def-varphi}, respectively.
Then there exists a constant $\psi_2 \ge 0$ such that 
$$
|\Psi(x, y)| \le \psi_2
%\quad\text{ for all } \quad (x, y)\in \R^d\times\R^d\setminus Z
.
$$
\end{lem}
\begin{proof}
For $k=1$,
using property~\eqref{ass:K7} of $h'$
and 
property \eqref{eq:phi-prop2} for $\varphi'$,
all terms in definition~\eqref{Phi:A} of the function~$\Phi$ are bounded (and thus so is $\Psi$).

For $k \in (0,1)$, by definition~\eqref{Psi:A}, it is sufficient to show that the function $\Phi$ is bounded.
First, we estimate $\Phi$ on the set
\begin{equation}\label{set1}
\Omega_1\equiv
\left\{
    (x, y) \in \R^d\times \R^d \setminus Y : \ |x| \le \frac{3}{2} \quad \text{and} \quad |y| \le \frac{3}{2}
\right\}.
\end{equation}
Consider formula~\eqref{Phi:A} for $\Phi$. We add and subtract the term
    $$
    \varphi' ( |x|^{2-k} ) \frac{y}{|y|^k}
    \cdot
    \frac{x-y}{|x-y|^{2-k}} h'(|x-y|^k)
    $$
and rearrange the terms in the following way:
\begin{equation}\label{Phi:decomp}
    \begin{split}
    \Phi(x,y)=&
    \left[
        \varphi'(|x|^{2-k})
        \left(\frac{x}{|x|^{k}}-\frac{y} {|y|^{k}}\right)
    \right.
    \\ & \quad
    +
    \left.
        \frac{y}{|y|^{k}}
        \left(\varphi^{\prime}(|x|^{2-k})
        - \varphi^{\prime}(|y|^{2-k})\right)
    \right]
    \cdot\frac{x-y}{|x-y|^{2-k}} h^{\prime}(|x-y|^{k}).
    \end{split}
\end{equation}
The function $\varphi''$ is globally bounded, which follows from its properties in Remark~\ref{rem:def-varphi}.
Thus
\begin{equation}\label{eq:varphi-diff}
    \begin{split}
    \left|
        \varphi'(|x|^{2-k}) - \varphi'(|y|^{2-k})
    \right|
    & \le \|\varphi''\|_\infty
    \left|
        |x|^{2-k}-|y|^{2-k}
    \right| \\
    & \le (2-k) \left( \frac{3}{2} \right)^{1-k} \|\varphi''\|_\infty |x-y|,
    \end{split}
\end{equation}
as $2-k>1$ and $|x|, |y|\le 3/2$. Now we can estimate each term in formula~\eqref{Phi:decomp}:
\begin{equation}
    \begin{split}
        |\Phi(x,y)| \le &
        \left[
            2^k
            + 
            (2-k) \left( \frac{3}{2} \right)^{1-k} \|\varphi''\|_\infty
            |y|^{1-k}
            |x-y|^k
        \right]
        \|h'\|_\infty,
    \end{split}
\end{equation}
where we used
properties of the function $\varphi$ from Remark~\ref{rem:def-varphi},
Lemma~\ref{lem:Holder-2k} with $\beta = k \in (0,1)$,
estimate~\eqref{eq:varphi-diff} and
property~\eqref{ass:K7} of the function $h$. 
The right-hand side of this estimate is bounded on the set $\Omega_1$. % from definition~\eqref{set1}.

Next, by the properties of the function $\varphi$ in~\eqref{eq:phi-prop3}, we have
$
\Phi(x,y)=0
$
on
$$
\Omega_2 \equiv
\left\{
    (x, y) \in \R^d\times\R^d \setminus Y : \ |x| > \frac{3}{2} \quad \text{and} \quad |y|>\frac{3}{2}
\right\}.
$$

It remains to show that $\Phi$ is bounded on 
$
\R^d\times\R^d
\setminus
( \Omega_1 \cup \Omega_2)
$.
We consider $\Phi$ on the set 
\begin{equation}
\Omega_3 \equiv 
\left\{
    (x, y) \in \R^d\times\R^d \setminus Y: \ |x|>\frac{3}{2} \quad \text{and} \quad |y|\le\frac{3}{2}
\right\},
\end{equation}
where, by \eqref{eq:phi-prop3}, we have
\begin{equation} \label{Phi:onOmega3}
\Phi(x,y) =
    -\varphi' ( |y|^{2-k} ) \frac{y}{|y|^k}
    \cdot
    \frac{x-y}{|x-y|^{2-k}}h'(|x-y|^k).
\end{equation}
Since $k\in (0,1)$ and $\varphi', h'\in L^\infty (0,\infty)$,
we immediately obtain 
$$\sup_{|y|\le 3/2}
|\Phi(x,y)| \to 0
\quad \text{as} \quad
|x| \to +\infty.$$
In order to estimate $\Phi(x,y)$ in a neighborhood of the set 
\begin{equation}\label{set32}
\left\{
(x,y) \in \R^d \times \R^d : \
|x|=\frac{3}{2}\quad  \text{and} \quad |y|=\frac{3}{2}
\right\},
\end{equation}
we use an estimate analogous to~\eqref{eq:varphi-diff}.
% we note that the function $\varphi'(|y|^{2-k})$ is Lipschitz continuous in a neighborhood of the sphere $\{y \in \R^d : \ |z|=3/2\}$. 
Since $|x| > 3/2$ on $\Omega_3$, we have $\varphi'(|x|^{2-k}) = 0$. Thus, in this neighborhood, we obtain
$$
\left| \varphi'(|y|^{2-k}) \right|
= 
\left|\varphi'(|y|^{2-k}) -\varphi'(|x|^{2-k}) \right| \le C |x - y|.
$$

Consequently, the function in~\eqref{Phi:onOmega3} is bounded on the set $\Omega_3$. Analogous estimates hold on the set 
\begin{equation}
\Omega_4 \equiv 
\left\{
    (x, y) \in \R^d\times\R^d  \setminus Y : \
    |x| \le \frac{3}{2}
    \quad \text{and} \quad
    |y| > \frac{3}{2}
\right\},
\end{equation}
which completes the proof that the function~$\Phi$ is bounded on $\R^d\times\R^d \setminus Y $ for $k\in (0,1)$.

For $k \in (1,2)$, the function $|\Phi|$ is no longer bounded due to the singularity induced by the exponent $1 - k < 0$.
Consequently, we consider the function $\Psi$ defined in~\eqref{Psi:A} and apply the substitution~\eqref{eq:subst-zw} by analogy with the proof of Lemma~\ref{lem:1A}.
In these new variables, we write $\Psi(x,y) = \widetilde\Psi(z,w)$, where
\begin{equation}
    \begin{split}
        \widetilde\Psi(z,w) = &
        \ 2^{k-1} 
        \left(
            h'(|z|^k)\frac{z}{|z|^{2-k}} - h'(|w|^k) \frac{w}{|w|^{2-k}}
        \right)
        \cdot
        \frac{z-w}{|z-w|^k} \varphi'
        \left( 2^{k-2}|z-w|^{2-k} \right) \\
        & +
        2^{k-1} 
        \left(
            h'(|z|^k)\frac{z}{|z|^{2-k}} + h'(|w|^k) \frac{w}{|w|^{2-k}}
        \right)
        \cdot
        \frac{z+w}{|z+w|^k} \varphi' 
        \left( 2^{k-2}|z+w|^{2-k} \right).
    \end{split}
\end{equation}
Note that $\widetilde\Psi(z,w)$ has the same structure (up to multiplicative constants) as $\Psi(x,y)$, but with the exponent $k$ replaced by $2-k \in (0,1)$ and the roles of $\varphi'$ and $h'$ interchanged.
Thus, to complete the proof, it suffices to follow the reasoning from the first part, where $\Psi$ is analyzed for $k \in (0,1)$. Moreover, besides using the properties of $\varphi$, we rely on the 
facts that $h' \in L^\infty(0,\infty)$ by property~\eqref{ass:K7}, and 
$h'', |\cdot|^kh'' \in L^\infty(0,\infty)$ for $k \in (1,2)$, which is guaranteed by 
assumption~\eqref{ass:K6} and property~\eqref{ass:K10}.
Below, we sketch the main steps of that analysis.

The function $\tilde{\Psi}$ is bounded on any bounded set following the study of $\Psi$ over $\Omega_1$.
It  vanishes when $|z-w|$ and $|z+w|$ are large enough.
Now consider the case when $z$ and $w$ are large but the distance between them is small enough so that $\varphi'
        \left( 2^{k-2}|z-w|^{2-k} \right)$ does not vanish.
%We use that by \eqref{ass:K10}, $|\cdot|^k h''\in L^\infty (0,\infty)$. CONFUSION OF FIRST AND SECOND DERIVATIVE NEEDS TO GET FIXED
%Namely ($C>0$ denotes constants depending on $h,k$;\ $|z|$ and $|w|$ have the same size up to an additive or multiplicative constant), 
Applying  property \eqref{ass:K10} and Lemma~\ref{lem:Holder-2k} and differentiating the function $x \mapsto x^k$ we get
\begin{align}
|\tilde{\Psi}(z,w)| &\le 
C \left( \left| h'(|z|^k)-h'|w|^k) \right| |z|^{k-1}+ 
\Big| \frac{z}{|z|^{2-k}}-\frac{w}{|w|^{2-k}} \Big| h'(|w|^k) \right) |z-w|^{1-k}
\\
& \le C \big||z|^k-|w|^k\big|\big||z|^k\big|^{-k} |z|^{k-1} |z-w|^{1-k}+ \left| \frac{z}{|z|^{2-k}}-\frac{w}{|w|^{2-k}} \right| |z-w|^{1-k}
\\
& \le C |z|^{k-1} \big||z|-|w|\big|^{k-1} |z|^{k-1-k^2} |z-w|^{1-k}   + C |z-w|^{k-1} |z-w|^{1-k}
\\
& \le  C |z|^{-(k-1)^2-1}+C \le C.
\end{align}
    \item The case when $z$ and $-w$ are large but the distance between them is small enough so that $\varphi'
        \left( 2^{k-2}|z+w|^{2-k} \right)$ does not vanish is treated in the same way.
\end{proof}

\section*{Acknowledgments}%\retka{Anything else?}
The third author has been supported by the project ANR SMASH 25-CE40-4532.
%The authors would like to express their gratitude to Piotr Biler and Lorenzo Brandolese for fruitful discussions.

\section*{Declarations}

\subsection*{Data Availability Statement}
Data sharing is not applicable to this article as no datasets were generated or analyzed during the current study.

\subsection*{Conflict of Interest Statement}
The authors declare that they have no competing interests.

\bibliographystyle{siam}
\bibliography{KKL-biblio}

\end{document}